\pdfoutput=1
\documentclass[11pt,reqno]{amsart}
\usepackage[a4paper,margin=1in,footskip=0.3in]{geometry}
\usepackage[table]{xcolor}
\usepackage{mathrsfs}
\usepackage{amssymb}
\usepackage{mathtools}
\usepackage{tikz-cd}
\usepackage{enumitem}
\usepackage[all,cmtip]{xy}
\usepackage[square, comma, numbers, sort&compress]{natbib}
\usepackage[T1]{fontenc}
\usepackage{booktabs}
\usepackage{pgf,tikz}
\usepackage{tikz-cd}
\usepackage{subcaption}
\usepackage{setspace}
\usepackage{csquotes}
\usepackage{epigraph}
\usepackage{extarrows}
\usetikzlibrary{arrows}
\definecolor{qqqqff}{rgb}{0,0,1}
\usepackage{lscape}
\usepackage{bm}
\usepackage{todonotes} 
\usepackage[normalem]{ulem}

\PassOptionsToPackage{pdfusetitle,pagebackref,colorlinks}{hyperref}
\usepackage{footnotebackref}

\usepackage{bookmark}
\hypersetup{
  linkcolor={blue!70!black},
  citecolor={red!60!black},
  urlcolor={blue!80!black}
}
\usepackage[capitalize]{cleveref}
\newtheorem{theorem}{Theorem}[section]
\newtheorem{lemma}[theorem]{Lemma}
\newtheorem{proposition}[theorem]{Proposition}
\newtheorem{corollary}[theorem]{Corollary}

\newtheorem{alphtheorem}{Theorem}

\theoremstyle{definition}
\newtheorem{definition}[theorem]{Definition}

\newtheorem{example}[theorem]{Example}

\theoremstyle{remark}
\newtheorem{remark}[theorem]{Remark}

\newtheoremstyle{cited}{.5\baselineskip\@plus.2\baselineskip\@minus.2\baselineskip}{.5\baselineskip\@plus.2\baselineskip\@minus.2\baselineskip}{\itshape}{}{\bfseries}{\bfseries .}{5pt plus 1pt minus 1pt}{\thmname{#1}\thmnumber{~#2}\thmnote{ \normalfont#3}}
\theoremstyle{cited}
\newtheorem{citedthm}[theorem]{Theorem}

\newtheorem{citedprop}[theorem]{Proposition}

\newtheoremstyle{citeddef}{.5\baselineskip\@plus.2\baselineskip\@minus.2\baselineskip}{.5\baselineskip\@plus.2\baselineskip\@minus.2\baselineskip}{}{}{\bfseries}{\bfseries .}{5pt plus 1pt minus 1pt}{\thmname{#1}\thmnumber{~#2}\thmnote{ \normalfont#3}}
\theoremstyle{citeddef}

\newcommand{\mor}[1][]{\xrightarrow{#1}}
\newcommand{\isomor}{\mor[\sim]}

\newcommand\hz{\mathbb Z}
\newcommand\Z{\mathbb Z}

\def\GL{\mathop{\mathrm{GL}}\nolimits}

\newcommand{\Hom}{\text{Hom}}
\def\codim{\mathop{\mathrm{codim}}\nolimits}

\newcommand{\rk}{\operatorname{rank}}

\newcommand{\NS}{\operatorname{NS}}
\newcommand{\Db}{\textbf{D}^\text{b}}

\newcommand{\tr}{\texttt{tr}}

\newcommand{\Pp}{\mathcal{P}}

\newcommand{\dive}{\text{div}}

\newcommand{\Exc}{\text{Exc }}

\def\wGL2{\ensuremath{\widetilde{\mathrm{GL}}_2^+(\bbR)}}
\newcommand{\ign}[1]{}

\newcommand{\Span}{\operatorname{span}}

\newcommand{\Br}{\operatorname{Br}}

\newcommand{\ord}{\operatorname{ord}}

\newcommand{\MonH}{\operatorname{Mon}_{\rm Hdg}^2}

\DeclareMathOperator{\Aut}{Aut}

\DeclareMathOperator{\im}{Im}

\DeclareMathOperator{\Pic}{Pic}

\DeclareMathOperator{\Hilb}{Hilb}
\DeclareMathOperator{\Alb}{Alb}

\DeclareMathOperator{\Mov}{Mov}

\DeclareMathOperator{\Amp}{Amp}
\DeclareMathOperator{\Nef}{Nef}
\DeclareMathOperator{\Mon}{Mon}

\DeclareMathOperator{\Bir}{Bir}
\DeclareMathOperator{\Kum}{Kum}

\DeclareMathOperator{\Ext}{Ext}

\DeclareMathOperator{\ext}{ext}

\DeclareMathOperator{\ch}{ch}

\DeclareMathOperator{\Stab}{Stab}
\DeclareMathOperator{\Stabd}{Stab^\dagger}

\DeclareMathOperator{\Coh}{Coh} %Actually better not bold

\newcommand\calA{\mathcal{A}}

\newcommand\calC{\mathcal{C}}
\newcommand\calD{\mathcal{D}}
\newcommand\calE{\mathcal{E}}
\newcommand\calF{\mathcal{F}}

\newcommand\calH{\mathcal{H}}

\newcommand\calK{\mathcal{K}}

\newcommand\calM{\mathcal{M}}

\newcommand\calO{\mathcal{O}}
\newcommand\calP{\mathcal{P}}

\newcommand\calT{\mathcal{T}}

\newcommand\calW{\mathcal{W}}

\newcommand{\sA}{\mathcal{A}}

\newcommand{\sE}{\mathcal{E}}

\newcommand{\sH}{\mathcal{H}}

\newcommand{\sK}{\mathcal{K}}

\newcommand{\sM}{\mathcal{M}}

\newcommand{\sO}{\mathcal{O}}

\newcommand{\sW}{\mathcal{W}}

\newcommand{\bbC}{\mathbb{C}}
\newcommand{\bbG}{\mathbb{G}}
\newcommand{\bbH}{\mathbb{H}}

\newcommand{\bbN}{\mathbb{N}}
\newcommand{\bbP}{\mathbb{P}}
\newcommand{\bbQ}{\mathbb{Q}}
\newcommand{\bbR}{\mathbb{R}}
\newcommand{\bbZ}{\mathbb{Z}}

\newcommand{\fraka}{\mathfrak{a}}

\newcommand{\into}{\hookrightarrow}

\DeclareMathOperator{\dv}{div}
\DeclareMathOperator{\Id}{Id}

\title[Finite Order Maps on Kummer-type Manifolds]{Finite Order Symplectic Birational Self-maps on Kummer-type Manifolds}
\date{}
\author{Yajnaseni Dutta}
\address{Universiteit Leiden, Mathematisch Instituut, Gorlaeus Gebouw, Einsteinweg 55, 2333 CC, the Netherlands.}
\email{y.dutta@math.leidenuniv.nl}
\author{Dominique Mattei}
\address{Institut f\"{u}r Algebraische Geometrie, Leibniz Universit\"{a}t Hannover, Welfengarten 1, 30167 Hannover, Germany.}
\email{mattei@math.uni-hannover.de}
\email{muller@math.uni-hannover.de}
\author{Stevell Muller}
\author{Howard Nuer}
\address{Technion - Israel Institute of Technology, Derech Ya'akov Dori, Haifa, 3200003, Israel}
\email{hnuer@technion.ac.il}

\usepackage{bookmark}

\begin{document}
\subjclass[2020]{} 
\keywords{}

    \begin{abstract}
    A projective hyperk\"ahler manifold of Kummer-type is said to be twisted modular if it is birational to the Albanese fiber of a moduli space of twisted sheaves on an abelian surface. We prove that, with the exception of certain cases of Picard rank 3, any projective Kummer-type manifold admitting a finite-order symplectic birational self-map that acts nontrivially on its second cohomology group is twisted modular. We provide a complete characterization of these exceptions in terms of their N\'eron--Severi lattices. We then investigate symplectic birational self-maps of modular Kummer-type manifolds, determining exactly which Mukai vectors allow the birational transformation induced by crossing the vertical wall, which acts on cohomology as a reflection, to correspond to a finite-order symplectic birational self-map. Additionally, we prove in an appendix several results concerning moduli spaces of twisted sheaves on abelian surfaces which were not readily available in the literature.
\end{abstract}
 \maketitle
\setcounter{tocdepth}{1}
\tableofcontents
\hypersetup{bookmarksdepth=2}

Hyperk\"ahler manifolds occupy a central role in the classification of compact K\"ahler manifolds with trivial canonical bundle \cite{BogDecomp,Bea83}. In this paper, we study the symmetries of hyperk\"ahler manifolds of $\Kum_n$-type, a fundamental series of examples introduced by Beauville \cite{Bea83}. For an abelian surface $A$ and an integer $n\ge1$, the generalized Kummer manifold $K_n(A)$ of dimension $2n$ is defined as the fiber over zero of the summation map $\Sigma\colon A^{[n+1]}\to A$ from the Hilbert scheme of points to the abelian surface. Any hyperk\"ahler manifold deformation equivalent to $K_n(A)$ is said to be of $\Kum_n$-type, or Kummer-type when specifying $n$ is not necessary.

Our goal is to establish geometric constraints for the existence of finite birational self-maps on such manifolds (Theorem \ref{thm:main theorem}), and to provide explicit examples of such self-maps on modular $\Kum_n$-type manifolds (Theorem \ref{theorem:main theorem 2-explicit examples}). These results rely on a criterion for studying birational isomorphisms between $\Kum_n$-type manifolds, which is of independent interest. Building on the work of Wieneck \cite{Wie18}, for any $\Kum_n$-type manifold $X$, we define a canonical $O(U^{\oplus 4})$-orbit of embeddings $\iota_X\colon H^2(X,\bbZ) \to U^{\oplus 4}$, where $U$ denotes the hyperbolic plane lattice. We extend the Hodge structure of $H^2(X,\bbZ)$ to $U^{\oplus 4}$ by requiring that $H^2(X,\bbZ)^\perp\subset U^{\oplus 4}$ is of type $(1,1)$.

\begin{alphtheorem}[see Theorem~\ref{thm:biratU4}]\label{thm:IntroThmCriterion}
Let $X$ and $Y$ be $\Kum_n$-type manifolds for some $n\geq 2$, and let $\iota_X,\iota_Y$ be two compatible (see Definition \ref{def:compatible embeddings}) Wieneck embeddings. The following are equivalent:
\begin{enumerate}
    \item[(i)] $X$ and $Y$ are bimeromorphic,
    \item[(ii)] There exists a commutative diagram
    $$\begin{tikzcd}
        H^2(X,\bbZ) \ar[r,hook,"\iota_X"] \ar[d,"g"] & U^{\oplus 4} \ar[d,"\widetilde{g}"] \\
        H^2(Y,\bbZ) \ar[r,hook,"\iota_Y"] & U^{\oplus 4}
    \end{tikzcd}$$
    such that $\widetilde{g}\in SO(U^{\oplus 4})$, and $g$ is a Hodge isometry.
\end{enumerate}
\end{alphtheorem}

This criterion is already known for hyperk\"ahler manifolds of $\text{K3}^{[n]}$-type thanks to the work of Markman \cite{Marsurvey}. However, the requirement regarding the sign of $\det(\widetilde{g})$ in Theorem \ref{thm:IntroThmCriterion} is a subtlety unique to $\Kum_n$-type manifolds and is necessary (see Example \ref{ex:nonMon}).

Recent developments in the study of hyperk\"ahler manifolds and their singular analogs motivate the analysis of finite symplectic birational actions. For instance, if a finite group $G$ acts symplectically on $X$, any symplectic form $\sigma_X\in H^{2,0}(X)$ restricts to the fixed locus $X^G$ and descends to the quotient $X/G$. 
In both contexts, understanding the geometry of the $G$-action on $X$ often gives rise to numerous examples of primitive symplectic varieties \cite{zbMATH08068450,arXiv:2604.06851,arXiv:2308.14692,arXiv:2211.14524}.

In \cite{DMPM}, the first two named authors and Prieto-Monta\~nez demonstrated that any projective $\text{K3}^{[n]}$-type hyperk\"ahler manifold admitting a finite-order symplectic birational self-map is twisted modular. We say that a projective hyperk\"ahler manifold is \emph{twisted modular} if it is birational to the Albanese fiber of a moduli space of stable twisted sheaves on a $K$-trivial surface. When the surface is K3, this construction yields manifolds of $\text{K3}^{[n]}$-type (where the Albanese fiber is the entire moduli space). When the surface is abelian, one obtains the Kummer-type manifolds studied here. Our aim is to determine the extent to which the results of \cite{DMPM} apply to Kummer-type manifolds, and to highlight where they diverge. We prove the following:

\begin{alphtheorem}[see Theorem~\ref{thm:modular except for some examples}]\label{thm:main theorem}
Let $X$ be a projective $\Kum_n$-type manifold for some $n\geq2$. Assume there exists a symplectic birational self-map $f$ of finite order with a non-trivial action on $H^2(X,\bbZ)$. Then $X$ is twisted modular unless there exists a positive integer $d\geq 1$ such that either of the following holds:
\begin{enumerate}
    \item[\textnormal{(I)}] There is a finite index embedding $\langle2d\rangle\oplus A_1^{\oplus2}(-1)\hookrightarrow \NS(X)$ and $(n+1)x^2+dy^2-z^2-t^2$ is anisotropic.
    \item[\textnormal{(II)}] There is a finite index embedding $\langle2d\rangle\oplus A_2(-1)\hookrightarrow \NS(X)$ and $(n+1)x^2+dy^2-3z^2-t^2$ is anisotropic.
\end{enumerate}
Moreover, both cases (I) and (II) can occur, and Corollary~\ref{cor: NeronSeveriExceptions} lists all the possible isometry classes for $\NS(X)$ in these cases.
\end{alphtheorem}

\begin{remark}
    Unlike the $\text{K3}^{[n]}$ case, where the representation $\Bir(X)\to \GL(H^2(X, \bbZ))$ is faithful, $\Kum_n$-type manifolds have a representation with a nontrivial kernel. Consequently, they always admit finite-order symplectic self-maps with a trivial action on $H^2(X,\bbZ)$. 
\end{remark}

Regarding the non-modular exceptions identified in Theorem \ref{thm:main theorem} (I) and (II), it would be interesting to find a concrete geometric realization for these symmetries. 

\begin{remark}
Note that the exceptions (I) and (II) in Theorem~\ref{thm:main theorem} appear in Picard rank $3$. In fact, in the proof of the theorem we show if $X$ is of Picard rank $2$ or more than $3$ and there exists a symplectic birational self-map $f$ of finite order with a non-trivial action on $H^2(X,\bbZ)$, then $X$ is twisted modular.
 By contrast, already in Picard rank $2$,  the existence of an \emph{infinite}-order symplectic birational self-map certainly does not ensure modularity. For instance, 
 if $X$ is a $\Kum_2$-type manifold such that $\NS(X)$ has the Gram matrix $\begin{pmatrix} 4&6\\6&4 \end{pmatrix}$ with respect to some basis, then $X$ admits an infinite-order symplectic birational self-map whose action on $\NS(X)$ is given by $\begin{pmatrix} -55&144\\-144 &377 \end{pmatrix}$. By our numerical criterion, however, such an $X$ is not twisted modular. 
 This example arose out of a conversation between the third-named author and Simone Billi.
\end{remark}

Theorem \ref{thm:main theorem} suggests that the search for finite-order birational self-maps should be focused on twisted modular Kummer-type manifolds. Following the strategy of \cite{DMPM}, we let $A$ be an abelian surface and $X=K_A(v)$ be a modular Kummer-type manifold. The stability manifold $\Stab(A)$, alongside its wall-and-chamber decomposition determined by $v$, controls the birational properties of $X$. By isolating a specific wall-crossing in $\Stab(A)$, we obtain the following result:

\begin{alphtheorem}\label{theorem:main theorem 2-explicit examples}
Let $A$ be an abelian surface and $v=(r,cD,\chi)$ be a primitive Mukai vector with $D\in\NS(A)$ primitive and $r\ne 0$. Consider the Albanese fiber $K:=K_A(v,\sigma)$ of the moduli space of stable objects with respect to a $v$-generic stability condition $\sigma$. For $e=(r,cD,\frac{c^2D^2}{r}-\chi)$, the reflection $R_e\in O(H^2(K,\bbZ))$ is induced by a symplectic birational self-map if and only if $r\mid 2c$, $\gcd(r,\chi)=1$ or $2$, $r\ne 1$ or $2$, and $v$ does not belong to one of the following series of Mukai vectors:
\begin{enumerate}
    \item $(r,krD,\frac{k^2D^2r}{2}-m)$ for some $k\in\bbZ$ and $m\in \{1, 2\}$.
    \item $D^2\equiv 0\pmod 4$, $v=(2a,maD,\frac{D^2}{4}m^2a-1)$ for integers $a\ge 2$ and odd $m$.
    \item $D^2\equiv 2\pmod 4$, $v=(2a,maD,\frac{D^2m^2a-2}{4})$ for odd integers $a\ge 3$ and $m$.
\end{enumerate}
\end{alphtheorem}

Theorem \ref{theorem:main theorem 2-explicit examples} provides infinitely many examples of birational self-maps on Kummer-type manifolds that act as reflections on the second cohomology group. 
Furthermore, we provide an explicit geometric interpretation of these maps in Example~\ref{example:MarkmanExample}.
For technical simplicity, we restricted our focus to non-twisted moduli spaces in this theorem, though we expect a parallel result to hold in the twisted case. 

Finally, we address a gap in the literature regarding moduli spaces of \emph{twisted} sheaves on abelian surfaces. While many of these results are known to the experts, precise references are scarce. We review and sometimes reprove some results in Appendix~\ref{sec:modulispaces}, see Theorem~\ref{thm:twistedYoshioka} and Corollary~\ref{coro:mukai+marking=wieneck}.
Even though the techniques are similar to the case of K3 surfaces, it turns out that the extension to abelian surfaces is more subtle.
Therefore, we provide missing details as well as many references  where necessary.

\subsection*{Structure of the paper}

In Section \ref{sec:prelim}, we gather preliminaries on lattices,  cones in hyperkähler geometry
and Hodge Torelli theorems. In Section \ref{sec:KummerHK}, we recall examples and properties of Kummer-type manifolds, detail the construction and properties of Wieneck embeddings, and prove Theorem \ref{thm:IntroThmCriterion}. In Section \ref{sec:SympBirat}, we study finite-order symplectic birational self-maps, analyze the non-modular symmetric examples, and prove Theorem \ref{thm:main theorem}. Finally, Section \ref{sec:RefVertWall} is devoted to Bridgeland stability conditions on abelian surfaces, the classification of the vertical wall, and the proof of Theorem \ref{theorem:main theorem 2-explicit examples}. Appendix \ref{sec:modulispaces}  
fills some gaps in the literature concerning moduli spaces of twisted sheaves on abelian surfaces.

\subsection*{Acknowledgment}
We are grateful to Jonathan Love, Eyal Markman, Ludvig Modin, Giovanni Mongardi, Emanuel Reinecke, and Paolo Stellari for patiently answering our questions, to Simon Brandhorst for the discussions on wall divisors, and to Martijn Kool for making us aware of \cite{vBGJM}.
The authors would also like to thank Stefan Schreieder and the Institut für Algebraische Geometrie at the Leibniz Universität Hannover for supporting a productive work visit and working environment that facilitated the final push resulting in this paper.
HN was partially supported by ISF Grant 3175/25.

\section{Preliminaries}\label{sec:prelim}

\subsection{Lattices}
A \emph{lattice} $L$ is a finitely generated free $\bbZ$-module equipped with a $\bbQ$-valued nondegenerate symmetric bilinear form $(\cdot,\cdot)$. Equivalently, we can define a lattice $L$ by the quadratic form $x\mapsto (x,x)$. 
We call $L$ \emph{integral} if $(x,y)\in\bbZ$ for all $x,y\in L$ and \emph{even}, if moreover $x^2\coloneqq (x,x) \in2\mathbb{Z}$ for all $x\in L$. An integral lattice $L$ is said to be \emph{indivisible} if there exist $x,y\in L$ such that $(x,y) = 1$.

Given a lattice $L$, we denote $L_R := L\otimes_{\mathbb{Z}}R$ for any $R\in \{\bbQ, \bbR, \bbC\}$, which we equip with a bilinear form by extending $(\cdot,\cdot)$ bilinearly.
We define the \emph{signature} of $L$ to be the signature $(l_+, l_-)$ of the regular real quadratic space $L_\bbR$: The lattice $L$ is called \emph{positive definite} (resp. \emph{negative definite}) if $l_- = 0$ (resp. if $l_+ = 0$) and \emph{hyperbolic} if $l_+ = 1$. In this paper, the ADE root lattices are assumed to be positive definite.
We denote by $U$ the hyperbolic plane lattice given by the Gram matrix
$\begin{pmatrix}
    0&1\\1&0
\end{pmatrix}$. 
We let moreover $L(-1)$ denote the same module $L$ but equipped with the opposite form $-(\cdot, \cdot)$.

Given an integral lattice $L$, we let $L^\vee:= \{x\in L_\mathbb{Q}\;\mid\; (x,y) \in \bbZ,\, \forall y\in L\}$ denote its dual lattice and we let $A_L := L^\vee/L$ be the so-called \emph{discriminant group} of $L$. If $L$ is even, then $A_L$ is equipped with a torsion quadratic form
\[q_L\colon A_L\to \bbQ/2\bbZ,\; x+L\mapsto x^2+2\bbZ.\]
For an integer $n\geq 1$, the lattice $L$ is said to be \emph{$n$-elementary} if $nA_L = \{0\}$ and $1$-elementary lattices are called \emph{unimodular} (they satisfy $L = L^\vee$).

Given a lattice $L$ and a sublattice $S\subseteq L$, we say that $S$ is \emph{primitive} in $L$ if $L/S$ has no torsion, and we call $L$ an \emph{overlattice} of $S$ if the quotient module $L/S$ is finite or, equivalently, if $\rk L = \rk S$. We call the lattice \[\operatorname{Sat}_L(S) := L\cap S_\mathbb{Q}\]
the \emph{saturation} of $S$ in $L$. By definition, the lattice $\operatorname{Sat}_L(S)$ is primitive in $L$ and it is the largest overlattice of $S$ contained in $L$. Since lattices are assumed to be nondegenerate, any morphism of lattices\footnote{By this we mean a morphism in the category of lattices, i.e. morphism of $\bbZ$-modules respecting bilinear forms.} $i\colon S\hookrightarrow L$ is an embedding. Such an embedding $i$ is said to be \emph{primitive} if $i(S)\subseteq L$ is a primitive sublattice and it is said to be \emph{of finite index} if $L$ is an overlattice if $i(S)$.

We denote by $O(L)$ the group of isometries of $L$. For a subgroup $G\subseteq O(L)$, we define:
\begin{enumerate}
    \item $L^G := \{v\in L\;\mid\; g(v)=v,\, \forall g\in G\}$ the \emph{invariant sublattice} of $G$,
    \item $L_G := (L^G)^{\perp}$ the \emph{coinvariant sublattice} of $G$,
    \item $SG := \{g\in G\;\mid\; \det(g) = +1\}$ the special subgroup of $G$.
\end{enumerate}
Similarly to Markman \cite[Section 4]{Marsurvey}, if $L$ has signature $(3, \ast)$, we denote by $G^+$ the subgroup of isometries of $G$ which preserve the orientation of a basis on any positive definite three-space $W\subseteq L_\mathbb{R}$. Any isometry in $O^+(L) := O(L)^+$ is said to be \emph{orientation-preserving}. For any isometry $f\in O(L)$, we denote similarly $L^f := L^{\langle f\rangle}$ and $L_f := L_{\langle f\rangle}$ the associated invariant and coinvariant sublattices.

Given a lattice $L$ and a nontrivial vector $v\in L\setminus\{0\}$, we say that $v$ is \emph{primitive} if the sublattice $\bbZ v\subseteq L$ is primitive. The vector $v$ is called \emph{isotropic} if $v^2 = 0$ and we call it \emph{positive} (resp.\ \emph{negative}) if $v^2 > 0$ (resp.\ $v^2 < 0$). The lattice $L$ is said to be \emph{isotropic} (resp.\ \emph{anisotropic}) if it admits an isotropic vector (resp.\ no isotropic vector). We define the \emph{divisibility} $\dive_L(v)$ of $v$ in $L$ to be the positive generator of the principal ideal $\{(v,y)\ \mid\ y\in L\}$. If $L$ is integral, then $\dive_L(v)$ is the largest positive integer $d$ such that $v\in dL^\vee$.

Given a lattice $L$ and a non-isotropic vector $v\in L_\bbR$, we define the \emph{reflection in $v$} to be the isometry $R_v\in O(L_\bbR)$ defined by:
\begin{equation}\label{eq:reflection}
R_v(x) := x - 2\frac{(x,v)}{v^2}v,\qquad \forall x\in L_\mathbb{R}.
\end{equation}
If $v\in L$ and $R_v$ defines an isometry of $L$, we call $v$ a \emph{root} of $L$. Note that if $L$ is integral and $v\in L$ is a vector, then $v$ is a root if and only if $v^2$ divides $2\dive_L(v)$. For any $n\in \mathbb{Z}$ nonzero, we call an \emph{$n$-vector} (resp.\ an \emph{$n$-root}) of $L$ any vector (resp.\ root) $v\in L$ such that $v^2 = n$.

Finally for $a_1,\ldots, a_n\in \bbQ^{\times}$, we denote by $\langle a_1,\ldots, a_n\rangle$ the lattice of rank $n$ whose Gram matrix in a given basis is diagonal with entries $a_1,\ldots, a_n$.

\subsection{Hyperk\"ahler manifolds, cones and Hodge Torelli}
Let $X$ be a \emph{hyperk\"ahler manifold}, that is a simply-connected compact K\"ahler manifold whose space of holomorphic two-forms is one-dimensional, spanned by a nowhere degenerate symplectic form. The second integral cohomology $H^2(X, \bbZ)$ of $X$ is equipped with the so-called \emph{Beauville--Bogomolov--Fujiki (BBF) quadratic form} $q_X$. It is an indivisible integral nondegenerate form, turning $(H^2(X, \mathbb{Z}), q_X)$ into a lattice of signature $(3, b_2(X)-3)$ \cite[Th\'eor\`eme 5]{Bea83}; for the rest of the paper, we always assume $H^2(X, \bbZ)$ comes equipped with its BBF-form and we omit $q_X$ from notation. We denote moreover by $\Mon^2(X)\subset O(H^2(X,\bbZ))$ the subgroup of \textit{monodromy operators}, that is, isometries induced by parallel transport operators (see \cite[Definition 1.1]{Marsurvey} for more detailed definitions), and $\Mon^2_{\rm Hdg}(X)\subset \Mon^2(X)$ those operators preserving the Hodge structure. In this paper, for any subring $R\subseteq \bbC$, we denote $H^{1,1}(X, R) := H^2(X, R)\cap H^{1,1}(X)$.\smallskip

The geometry of $X$ is controlled by several cones inside $H^{1,1}(X,\bbR)$, which admit certain wall-and-chamber decompositions. For this, we introduce two sets of classes in the N\'eron--Severi lattice $\NS(X) = H^{1,1}(X,\mathbb{Z})$ of $X$: the set of stably prime exceptional divisors $\calW^{\rm pex}_X$ and the set of wall divisors $\calW_X$ (see \cite[Sections 5 and 6]{Marsurvey} and \cite[Definition 1.2]{MonKahCone} for the definitions). They satisfy $\calW^{\rm pex}_X\subset \calW_X$. We define moreover the \emph{positive cone} $\calC_X$ of $X$ to be the connected component of $P_X:= \{x\in H^{1,1}(X,\bbR) \ \mid \ x^2>0\}$ containing a K\"ahler class.

\begin{remark}\label{rem monodromy are orientation-preserving}
    In \cite[Section 4]{Marsurvey}, Markman introduces the notion of orientation on $P_X$, related to the orientation of bases for any positive definite three-space in $H^2(X, \mathbb{R})$. It is known that a monodromy operator $f\in  \Mon^2(X)$ of $X$ preserves such an orientation, and in particular $f$ lies in $O^+(H^2(X, \bbZ))$, the subgroup of orientation-preserving isometries.
\end{remark}

\begin{definition}[{{{\cite[Definition 5.2]{Marsurvey}, \cite[Proposition 1.5]{MonKahCone}}}}]\label{def:wallchamberCX}
    We define the following convex sets:
    \begin{enumerate}
        \item An \textit{exceptional chamber} of $\calC_X$ is a connected component of
    \begin{equation}\label{eq:Excchamber}
        \calC_X \setminus \bigcup_{\alpha\in \calW^{\rm pex}_X}\alpha^\perp,
    \end{equation}
    and we call the exceptional chamber $\mathcal{FE}_X$  containing a K\"ahler class the \emph{fundamental exceptional chamber} $X$.
    \item A \textit{K\"ahler-type chamber} of $\calC_X$ is a connected component of
    \begin{equation}\label{eq:Kchamber}
        \calC_X \setminus \bigcup_{\delta\in \calW_X}\delta^\perp,
    \end{equation} 
    and we call the K\"ahler-type chamber $\mathcal{K}_X$ containing a K\"ahler class the \emph{K\"ahler cone} of $X$.
    \end{enumerate}
    There is an obvious inclusion $\mathcal{K}_X\subset \mathcal{FE}_X$.
\end{definition}

If $X$ is projective, we define moreover the following cones in $\overline{\calC_X}\cap \NS(X)_\bbR$:

\begin{enumerate}
    \item The \textit{ample cone} $\Amp(X)$ of ample classes, and its closure $\Nef(X)=\overline{\Amp(X)}$.
    \item The \textit{movable cone} $\Mov(X)$ which is the cone generated by classes of \textit{movable divisors}, i.e. divisors $D$ such that ${\rm codim} \ {\rm Bs}|D|\geq 2$.
\end{enumerate}

Defining these cones and chambers, we can formulate the Hodge-theoretic version of the Huybrechts--Verbitsky Torelli-type theorems for hyperk\"ahler manifolds.

\begin{theorem}[{\cite[Theorems 1.2, 1.6 and Corollary 5.7]{Marsurvey}}]\label{thm:BirmapHK}
    Let $X,Y$ be hyperk\"ahler manifolds which are deformation equivalent, and let $\phi\colon H^2(Y, \bbZ)\to H^2(X, \bbZ)$ be a parallel transport operator which is a Hodge isometry. 
    
    \begin{enumerate}
        \item There exists a bimeromophic map $f\colon X\dashrightarrow Y$ such that $\phi=f^\ast$ if and only if $\phi(\mathcal{FE}_Y) = \mathcal{FE}_X$;
        \item There exists a biholomorphic map $f\colon X\to Y$ such that $\phi = f^\ast$ if and only if $\phi(\calK_Y) = \calK_X$.
    \end{enumerate}
    If $X$ and $Y$ are projective, then one may replace fundamental exceptional chambers by movable cones and K\"ahler cones by amples cones in the previous two statements. Moreover, in this situation, one talks about birational maps and isomomorphisms.
\end{theorem}

According to \cite[Theorem 6.18]{Marsurvey}, any element $E\in \calW^{\rm pex}_X$ is a root of $H^2(X, \mathbb{Z})$, 
and the reflection  $R_E$ defines an element of $\Mon^2_{\rm Hdg}(X)$. Let $W_{\rm Exc}(X)\coloneqq \langle R_E \ \mid \ E\in \calW^{\rm pex}_X\rangle\subset \Mon^2_{\rm Hdg}(X)$ be the group generated by such reflections and we let $\Mon^2_{\rm Bir}(X)\subseteq \Mon^2_{\rm Hdg}(X)$ be the image of the natural map:
\[\rho_X\colon \Bir(X)\to O(H^2(X, \bbZ)),\]
where we denote by $\Bir(X)$ the group of bimeromorphic self-maps of $X$. The group $W_{\rm Exc}(X)$ is sometimes referred to as the \emph{Weyl group} of $X$.

\begin{theorem}[{{\cite[Theorem 6.18]{Marsurvey}}}]\label{theo:semidirect decomposition}
    The group $\Mon^2_{\rm Hdg}(X)$ acts transitively on the set of exceptional chambers of $\calC_X$ and the subgroup $W_{\rm Exc}(X)$ acts simply-transitively on this set. In particular $\overline{\mathcal{FE}}_X$ is a fundamental domain for the action of $\Mon^2_{\rm Hdg}(X)$ on $\calC_X$. There is  moreover a semidirect product decomposition
    $$\Mon^2_{\rm Hdg}(X)=W_{\rm Exc}(X)\rtimes \Mon^2_{\rm Bir}(X),$$
    and $\Mon^2_{\rm Bir}(X)$ coincides with the stabilizer of $\mathcal{FE}_X$ under the action of $\Mon^2_{\rm Hdg}(X)$.
\end{theorem}

\begin{remark}
     In this work, we are interested in the case where $X$ is projective. In particular, we want to work inside $\overline{\calC_X}\cap \NS(X)_\bbR$, rather than inside the full positive cone $\calC_X$. The interior of $\Mov(X)$ coincides with the convex hull of $\calF\calE_X\cap \NS(X)_\bbR$. Therefore, the wall-and-chamber decompositions in Definition~\ref{def:wallchamberCX} descend to $\overline{\calC_X}\cap \NS(X)_\bbR$. Applying similar terminology, we obtain that the interior of $\Mov(X)$ (resp.\ $\Amp(X)$) is an exceptional chamber (resp.\ K\"ahler-type chamber) of $\overline{\calC_X}\cap \NS(X)_\bbR$. We can therefore reformulate \Cref{theo:semidirect decomposition} by replacing $\calC_X$ by $\overline{\calC_X}\cap \NS(X)_\bbR$ and $\calF\calE_X$ by $\Mov(X)$.
\end{remark}

A combination of \Cref{thm:BirmapHK} and \Cref{theo:semidirect decomposition}, implies the following standard result.
\begin{corollary}\label{cor:Verbitskytorelli}
    Let $X,Y$ be hyperk\"ahler manifolds. Then they are bimeromorphic if and only if there exists a Hodge parallel transport isometry $H^2(X, \bbZ)\to H^2(Y, \bbZ)$.
\end{corollary}

\subsection{Moduli space of marked pairs and period map}
Let $X$ be a hyperk\"ahler manifold and let $\Lambda$ be a lattice isometric to $H^2(X, \mathbb{Z})$. 
Any fixed isometry $\mu\colon H^2(X, \bbZ)\to \Lambda$ is called a \emph{marking} of $X$, and we call the pair $(X, \mu)$ a \emph{$\Lambda$-marked} pair. Two such $\Lambda$-marked pairs $(X, \mu)$ and $(X', \mu')$ are said to be \emph{equivalent} if there is an isomorphism $f\colon X\to X'$ such that $\mu\circ f^\ast = \mu'$. For $X$ in a fixed deformation family, there is a coarse moduli space $\calM_{\Lambda}$ parametrizing $\Lambda$-marked pairs $(X, \mu)$ up to equivalence. 
This space is in general neither Hausdorff, nor connected.

\begin{proposition}[{{\cite[Theorem 4.6]{HuycptHKbasic}}}]\label{propo:same connected component gives parallel transport}
    Two equivalence classes $[(X, \mu)],\,[(X', \mu')]\in \calM_\Lambda$ are in the same connected component if and only if the isometry $\mu^{-1}\circ\mu'\colon H^2(X',\bbZ)\to H^2(X,\bbZ)$ is a parallel transport operator.
\end{proposition}
The space $\calM_\Lambda$ is equipped with the so-called \emph{period map}
\[\calP_{\Lambda}\colon \calM_{\Lambda}\to\Omega_{\Lambda},\; [(X, \mu)]\mapsto \mu(H^{2,0}(X))\]
whose target is the \emph{period domain}
\[\Omega_{\Lambda} := \left\{\bbC\sigma\in\bbP(\Lambda_\bbC)\;\mid\; \sigma^2 = 0,\, (\sigma,\overline{\sigma})>0\right\}.\]
We recall the following crucial result about the period map $\calP_\Lambda$.
\begin{proposition}[{{\cite[Theorem 8.1]{HuycptHKbasic}}}]\label{prop:surjectivity period map}
    Let $\calM^\circ_{\Lambda}$ be any connected component of the moduli space of $\Lambda$-marked pairs. Then, the restriction of  $\mathcal{P}_{\Lambda}$ to $\calM^\circ_{\Lambda}$ is surjective.
\end{proposition}

Let $\calM_{\Lambda}^\circ$ be a connected component of the moduli space of $\Lambda$-marked pairs. Then there is a set of vectors $\calW_\circ^{\rm pex}(\Lambda)$,  and a finite index subgroup $\Mon^2_\circ(\Lambda)$ of $O^+(\Lambda)$ satisfying the following: for any equivalence class $[(X, \mu)]\in \calM^\circ_{\Lambda}$, we have equalities
\begin{align}
    \calW^{\rm pex}_X &= \mu^{-1}\left(\calW_\circ^{\rm pex}(\Lambda)\right)\cap H^{1,1}(X, \mathbb{R})\label{eq:numerical walls}\\
    \Mon^2(X) &= \mu^{-1}\Mon^2_\circ(\Lambda)\mu.
\end{align}
The first equality follows from \cite[Proposition 5.14]{MarkmanPex} which tells us that stably prime exceptional classes are invariant under parallel transport (mapping some Hodge class to another Hodge class). The second equality is trivial and follows by definition of monodromy. Note that, a priori, the set $\calW_\circ^{\rm pex}(\Lambda)$ and the group $\Mon^2_\circ(\Lambda)$ may actually vary if we change the connected component $\calM_{\Lambda}^\circ$. However, in the case of Kummer-type manifolds, we recall in the next section that the description of the group $\Mon^2_{\circ}(\Lambda)$ is constant in $\calM_{\Lambda}$. For the rest of the paper, we refer to $\Mon^2_{\circ}(\Lambda)$ as the \emph{numerical monodromy group} and to the vectors in $\calW_\circ^{\rm pex}(\Lambda)$ as \emph{numerical stably prime exceptional divisors} (associated to a given connected component of $\calM_{\Lambda}$).

\section{Hyperk\"ahler manifolds of Kummer-type }\label{sec:KummerHK}

\subsection{Definition and examples}\label{sec:DefExamples}

Let $A$ be a complex torus of dimension 2 and $n\geq 1$. The summation map from the symmetric product $A^{(n+1)}\to A$, composed with the Hilbert--Chow resolution from the Douady space $A^{[n+1]}\to A^{(n+1)}$, gives a morphism
$$\Sigma\colon A^{[n+1]}\to A$$
which can be shown to be smooth, proper and isotrivial. Its fiber $K_n(A)=\Sigma^{-1}(0)$ is a hyperk\"ahler manifold of dimension $2n$ \cite[Th\'eor\`eme 4]{Bea83}, called a \textit{generalized Kummer manifold}. Any hyperk\"ahler  manifold $X$ of dimension $2n\geq 4$ which is deformation equivalent to such a Kummer-type manifold is called a \textit{(hyperk\"ahler) manifold of $\Kum_n$-type}, or just a \textit{Kummer-type manifold} if $n$ is not fixed. In this case, the second cohomology group $H^2(X,\bbZ)$, equipped with its BBF form, is isometric as a lattice to: 
$$U^{\oplus 3}\oplus \langle -2-2n \rangle,$$
see \cite[Proposition 8]{Bea83} for a proof.
We denote by $\delta$ a generator of $\langle -2-2n\rangle$.
Apart from generalized Kummer manifolds themselves, the simplest examples of $\Kum_n$-type manifolds arise from moduli spaces of (twisted) sheaves on abelian surfaces, which we now discuss.\smallskip

Let $A$ be an abelian surface, and consider a class $B\in H^2(A,\bbQ)$, called a \textit{B-field}. We define a $\mathbb{Z}$-polarized Hodge structure of weight 2 on $\widetilde{H}(A,B,\bbZ)$ as follows:
\begin{itemize}
    \item The underlying lattice is $H^{\rm even}(A,\bbZ)$ with Mukai pairing  
    $$((r,\theta,\chi),(r',\theta',\chi'))=\theta\cdot \theta'-r\chi'-\chi r',$$
    induced by the usual intersection pairing.
    \item The Hodge structure is defined by setting $\widetilde{H}^{2,0}(A,B)=\bbC(0,\sigma,\sigma\cdot B)$, where $\sigma\in H^{2,0}(A)$ is a symplectic generator.
\end{itemize}

We drop $B$ from the notation when it is trivial.

\begin{remark}\label{rmk:U(k)inH11}
There is always a finite index embedding
$$\bbZ(m,mB,0)\oplus \NS(A) \oplus \bbZ(0,0,1)\hookrightarrow \widetilde{H}^{1,1}(A,B,\bbZ),$$
where $m\in\bbZ$ is chosen such that $mB$ is integral. In particular one gets an embedding $U(m)\hookrightarrow \widetilde{H}^{1,1}(A,B,\bbZ)$. However, when $B\neq 0$, in general there is no embedding $U\hookrightarrow \widetilde{H}^{1,1}(A,B,\bbZ)$.
\end{remark}

\begin{definition}\label{def:posvector}
An integral vector $v=(r,\theta,\chi)\in \widetilde{H}^{1,1}(A,B,\bbZ)$ is called a \textit{Mukai vector}. It is called \textit{positive}\footnote{Here \emph{positive} does not mean that $v^2 > 0$ holds. This is a standard definition specific to Mukai vectors.} if one of the following holds: 
\begin{enumerate}
        \item[(i)] $r>0$,
        \item[(ii)] $r=0$, $\theta\in\NS(A)$ is effective,
        \item[(iii)] $r=\theta=0$, $\chi>0$.
    \end{enumerate}
\end{definition}
 Note that when $v^2\geq 2$, then either $v$ or $-v$ is positive. Indeed, if $r=0$, then $\theta^2 = v^2$ is positive and
the claim follows by Riemann--Roch. 

The following theorem concerns moduli spaces of stable $\alpha$-twisted sheaves on the abelian surface $A$, where $\alpha\in\Br(A)$ is a Brauer class. As $\Br(A)\simeq H^2(A,\calO_A^*)_{\rm tors}$ and $H^3(A,\bbZ)$ is torsion-free, it follows via the exponential exact sequence that $\alpha$ lifts to some $B^{0,2}\in H^2(A,\calO_A)$ which is the $(0,2)$-component of some rational cohomology class $B\in H^2(A,\bbQ)$ called a \emph{$B$-field lift} of $\alpha$.
The following theorem is proved in the non-twisted case, i.e. $\alpha=0$, by Yoshioka \cite[Theorem 0.2]{YosModuliSheaves}, and known by experts in the general twisted case, but we were not able to trace it in the literature. See Appendix~\ref{sec:modulispaces} for a short discussion about moduli spaces of twisted sheaves and a proof of the result.

For a primitive Mukai vector $v\in \widetilde{H}(A, \bbZ)$, we recall that a polarization $H\in \Amp(A)$ is called \emph{$v$-generic} if all $H$-semistable sheaves with Mukai vector $v$ are stable. A general polarization is $v$-generic, see e.g. \cite[Section 4.C]{HuyLeh} or \cite{PerRapISVModuli}, see also Proposition~\ref{prop:modulitwistsmoothproj}.

\begin{theorem}\label{thm:twistedYoshioka}
 Let $A$ be an abelian surface, $\alpha\in\Br(A)$ a Brauer class and $B\in H^2(A,\bbQ)$ a $B$-field of $\alpha$. Let $v=(r,\theta,\chi)\in \widetilde{H}(A,B,\bbZ)$ be a primitive positive Mukai vector with $v^2 = 2n+2\geq 6$, and $H$ a $v$-generic polarization. Then:
    \begin{enumerate}
        \item The moduli space $M_{A,\alpha}(v,H)$ of $H$-stable $\alpha$-twisted sheaves on $A$ with $B$-Mukai vector $v$ (see (\ref{eq:Bmukaiapp})) contains a smooth projective $2n$-dimensional hyperk\"ahler submanifold $K_{A,\alpha}(v,H)$ of $\Kum_{n}$-type as the fiber of its Albanese map.
        \item There exists an isometry of Hodge lattices
\begin{equation}\label{eq:MukaiMorphism}
    \theta_{v}^{-1}\colon H^2(K_{A,\alpha}(v,H),\bbZ) \xrightarrow{\sim} v^\perp \subset \widetilde{H}(A,B,\bbZ),
\end{equation}
called the \emph{Mukai morphism}, where the lattice structure on the LHS is induced by the BBF form.
    \end{enumerate} 
\end{theorem}

We prove this result in Appendix \ref{sec:modulispaces}.

\begin{definition}\label{def:modularKumn}
Hyperk\"ahler manifolds of $\Kum_n$-type arising from the construction of Theorem~\ref{thm:twistedYoshioka} are called \textit{twisted modular} (or simply \textit{modular} when the $B$-field vanishes).
\end{definition}

\subsection{Monodromy}
Let $X$ be a $\Kum_n$-type manifold for some $n\geq 2$. Recall that the monodromy group $\Mon^2(X)$ of $X$ is the subgroup of isometries induced by parallel transport operators, and we note in \Cref{rem monodromy are orientation-preserving} that $\Mon^2(X)\subseteq O^+(H^2(X,\bbZ))$, the group of orientation-preserving isometries.
As seen in Theorem~\ref{thm:BirmapHK}, isometries induced by birational self-maps are, in particular, monodromy operators.

\begin{citedprop}[{(\cite[Section 4]{MarMeh}, \cite[Theorem 2.3]{MongardiMonodromy})}]\label{prop:Mon2isN}
     The subgroup $\Mon^2(X)\subset O^+(H^2(X,\bbZ))$ consists of the isometries $f$ such that $\det(f)f$ acts trivially on $A_{H^2(X,\bbZ)}$.
\end{citedprop}

A typical candidate for such isometry is given by the reflection $R_e$ (see \eqref{eq:reflection}) in the hyperplane orthogonal to a root $e\in H^2(X,\bbZ)$.
We have the following result.

\begin{proposition}\label{Prop:Reflection in Mon}
    Let $e\in H^2(X,\bbZ)$ be a negative primitive root. Then the reflection $R_e$ belongs to $\Mon^2(X)$ if and only if $e$ satisfies 
    $$e^2=-2(n+1)\quad\text{and}\quad n+1\text{ divides }\dv_{H^2(X,\bbZ)}(e).$$
\end{proposition}
\begin{proof}
    Since any reflection is of determinant $-1$, the description of $\Mon^2(X)$ gives that $R_e$ is a monodromy if and only if it acts as 
    $-\Id$ on the discriminant group $A_{H^2(X,\bbZ)}$ of $H^2(X,\bbZ)$.

    If we suppose first that $R_e = -\Id$ on ${A_{H^2(X,\bbZ)}}$, we can apply \cite[Proposition 3.2(i)]{ghs07}: Since the exponent $D = 2n+2$ of $A_{H^2(X,\bbZ)}$ is even, such a reflection $R_e$ exists only if $e^2 = -D = -2(n+1)$ and $D/2 = n+1$ divides the divisibility of $e$ inside $H^2(X,\bbZ)$. This shows one direction. 

    Now we show the converse statement. Suppose that $e\in H^2(X,\bbZ)$ is a primitive class satisfying
        $$e^2=-2(n+1)\quad\text{and}\quad n+1\text{ divides }\dv_{H^2(X,\bbZ)}(e).$$
    Then by these assumptions, $R_e\in O^+(H^2(X,\bbZ))$.  To show it is in $\Mon^2(X)$ it remains to show that $R_e$ acts on $H^2(X,\bbZ)^\vee/H^2(X,\bbZ)$ via multiplication by $-1$.  Write $e=a+t\delta$ for $a\in U^{\oplus 3}$ and $t\in\bbZ$.  Then by the assumption that $n+1\mid \dv_{H^2(X,\bbZ)}(e)$ we get $a=(n+1)\xi$
    for some $\xi\in U^{\oplus 3}$.  Consider $w:=\frac{\delta}{2n+2}$.  Then $w+H^2(X,\bbZ)$ generates the discriminant group $A_{H^2(X,\bbZ)}$ which is isomorphic to $\bbZ/(2n+2)\bbZ$.  
    Thus 
    $$R_e(w)=w-\frac{2(w,e)}{e^2}e=w-\frac{t}{n+1}((n+1)\xi+t\delta)=-t\xi+(1-2t^2)w$$
    which is congruent to $(1-2t^2)w$ modulo $H^2(X,\bbZ)$.  
    The equality $e^2=-(2n+2)$ gives
    $$t^2-1=\frac{(n+1)(\xi,\xi)}{2},$$ 
    so $2t^2-2\equiv 0\pmod{2n+2}$, or equivalently $1-2t^2\equiv -1\pmod{2n+2}$. This implies that $R_e(w) + w \in H^2(X, \mathbb{Z})$, meaning that $R_e\in\Mon(X)$ as required.
\end{proof}

We conclude by remarking the following consequence of \Cref{prop:Mon2isN}. Let $\Lambda_n$ be a lattice isometric to $U^{\oplus3}\oplus \langle-2-2n\rangle$, where $n \geq 2$. For the rest of the paper, we fix that $\calM_{\Lambda_n}$ is the moduli space of $\Lambda_n$-marked pairs \emph{of $\Kum_n$-type}. For each connected component $\calM_{\Lambda_n}^\circ$ of the moduli space of $\Lambda_n$-marked pairs, the associated group $\Mon^2_\circ(\Lambda_n)$ coincides with
\[\Mon^2(\Lambda_n) := \left\{f\in O^+(\Lambda_n)\;\mid\;\det(f)f\text{ acts trivially on $A_{\Lambda_n}$}\right\}.\]
In particular, as claimed earlier, the definition of such a group does not depend on the choice of $\calM_{\Lambda_n}^\circ$. In what follows, we show how one can effectively determine elements of $\calW^{\rm pex}_\circ(\Lambda_n)$.

\subsection{Wieneck embeddings}
Let $\Lambda_n$ be a lattice isometric to $U^{\oplus3}\oplus \langle-2-2n\rangle$ for some fixed $n \geq 2$, and let $\widetilde{\Lambda}$ be isometric to $U^{\oplus4}$. The lattice $\widetilde{\Lambda}$ is unimodular, and for any abelian surface $A$, we have an identification of lattices: $\widetilde{H}(A,\bbZ)\simeq \widetilde{\Lambda}$. 
Recall that, $A_{\Lambda_n}\simeq \bbZ/(2n+2)\bbZ$.
The homomorphism $O(\Lambda_n)\to O(A_{\Lambda_n})$ is surjective \cite[Theorem 1.14.2]{nik79} and $\rk\widetilde{\Lambda} = \rk\Lambda_n+1$. Hence the analog of Witt's theorem in its strong form holds for the pair $(\widetilde{\Lambda}, \Lambda_n)$ \cite[Proposition 1.14.1]{nik79}. This means that $\Lambda_n$ admits at least one primitive embedding inside $\widetilde{\Lambda}$ and that there is a unique double coset in:
\[O(\widetilde{\Lambda})\backslash\{\Lambda_n\hookrightarrow\widetilde{\Lambda}\text{ primitive embedding}\}/O(\Lambda_n),\]
i.e. such a primitive embedding is unique up to actions of the orthogonal groups of both lattices.
The following lemma is well-known. 

\begin{lemma}\label{lem:monodromyextension}
    Fix any primitive embedding $\iota\colon \Lambda_n \to \widetilde{\Lambda}$ and pick $g\in O(\Lambda_n)$. Then, $g$ extends to an isometry $\widetilde{g}\in O(\widetilde{\Lambda})$ if and only if $g$ acts as $\pm \Id$ on $A_{\Lambda_n}$. Moreover, if any of the previous holds, $\widetilde{g}\in SO(\widetilde{\Lambda})$ if and only if $g$ acts as $\det(g)\Id$ on $A_{\Lambda_n}$.

\end{lemma}

\begin{proof}
    Note that there exists a primitive $(2n+2)$-vector $v\in\widetilde{\Lambda}$ such that $\iota(\Lambda_n)^\perp =\bbZ v\subset \widetilde{\Lambda}$. Now, since $O(\mathbb{Z}v) = \{\pm\Id\}$, the first statement follows from \cite[Corollary~1.5.2]{nik79} and the fact that $\widetilde{\Lambda}$ is unimodular. For the second statement, note that if $g$ acts as $\varepsilon \Id$ on $A_\Lambda$, then its extension $\widetilde{g}$ satisfies $\widetilde{g}(v)=\varepsilon v$. Since $\widetilde{\Lambda}_\bbQ=\Lambda_\bbQ\oplus \bbQ v$, the results follows.
\end{proof}

If one restricts to a right action of $\Mon^2(\Lambda_n)$, the number of double cosets of primitive embeddings $\Lambda_n\hookrightarrow \widetilde{\Lambda}$ as before grows. As we show in \Cref{lem: connected components and primitive embeddings}, each such double coset is linked to some connected component(s) of the moduli space $\calM_{\Lambda_n}$ of $\Lambda_n$-marked pairs of $\Kum_n$-type. But before that, we need to give a criterion for determining when certain isometries between the second cohomology groups of $\Kum_n$-type manifolds are induced by bimeromorphic maps. We start with the following construction.

\begin{citedthm}[{(\cite[Theorem 4.9]{Wie18})}]\label{thm:PrimEmbeddingH2}
Let $X$ be a $\Kum_n$-type manifold. There exists a canonical $\Mon^2(X)$-invariant $O(\widetilde{\Lambda})$-orbit of primitive embeddings $H^2(X,\bbZ) \hookrightarrow \widetilde{\Lambda}$. The orthogonal complement $H^2(X,\bbZ)^\perp\subset \widetilde{\Lambda}$ is generated by a primitive vector $v$ of square $2n+2$, and $\widetilde{\Lambda}$ inherits a Hodge structure from the one on $H^2(X, \bbZ)$, by declaring $v\in \widetilde{\Lambda}_\bbC^{1,1}$. 

If $X$ is twisted modular (Definition~\ref{def:modularKumn}), then the embedding $H^2(X,\hz)\hookrightarrow \widetilde{\Lambda}$ identifies, up to $O(\widetilde{\Lambda})$, with the Mukai morphism (\ref{eq:MukaiMorphism}).
\end{citedthm}

Let us briefly recall Wieneck's construction.  An embedding $\iota_X$ is chosen as 
\begin{equation}\label{eqn:naturalembedding}
    H^2(X,\bbZ) \xrightarrow{P} H^2(K_A(v, H),\bbZ) \xrightarrow{\theta_v^{-1}} v^\perp \hookrightarrow \widetilde{\Lambda}
\end{equation}
where $P$ is a parallel transport isometry between $X$ and some modular $\Kum_n$-type hyperk\"ahler manifold. Thanks to the work of Yoshioka \cite[Proof of Proposition~4.12 and Lemma~5.1]{YosModuliSheaves} (see also \cite[Remark~1.9]{PerRapISVModuli}), the $O(\widetilde{\Lambda})$-orbit of the embedding does not depend of any choice made. Indeed, any deformation between two modular $\Kum_n$-type manifolds $K_{A_1}(v_1, H_1)$ and $K_{A_2}(v_2,H_2)$ is monodromy equivalent to one induced by a deformation of the polarized surfaces $(A_1,H_1)$ to $(A_2,H_2)$, which is compatible with the Mukai pairing. Finally, we show in Appendix \ref{sec:modulispaces} that the same statement holds true in the twisted settings.

\begin{definition}\label{def:wieneck}
    For $X$ a $\Kum_n$-type manifold, we call any representative $\iota_X$ of the canonical $\Mon^2(X)$-invariant $O(\widetilde{\Lambda})$-orbit of primitive embeddings $H^2(X,\bbZ) \hookrightarrow \widetilde{\Lambda}$ from \Cref{thm:PrimEmbeddingH2} a \emph{Wieneck embedding of} $X$.
\end{definition}
Not any embedding $H^2(X,\bbZ)\into\widetilde{\Lambda}$ is a Wieneck embedding.
    In \cref{lem: connected components and primitive embeddings} we show that given an embedding $\kappa\colon \Lambda_n\into \Lambda$, there may be a marking $\mu\colon H^2(X,\bbZ)\to\Lambda_n$, for which the composition $\kappa\circ\mu$ is not a Wieneck embedding of $X$.

A direct corollary of \Cref{thm:PrimEmbeddingH2}, combined with Proposition~\ref{prop:Mon2isN} and \Cref{lem:monodromyextension}, is the following.

\begin{corollary}\label{cor:Mon2isSO}
    Let $X$ be a $\Kum_n$-type manifold. Fix a Wieneck embedding $\iota_X$ of $X$. An orientation-preserving isometry $g\in O^+(H^2(X,\bbZ))$ is a monodromy operator if and only if there exists $\widetilde{g}\in SO(\widetilde{\Lambda})$ such that $\widetilde{g}\circ \iota_X=\iota_X\circ g$.
\end{corollary}

In order to compare $\Kum_n$-type manifolds via Wieneck embeddings and isometries of $\widetilde{\Lambda}$, we generalize the previous result to parallel transport isometries between two $\Kum_n$-type manifolds.

\begin{definition}\label{def:compatible embeddings}
    Let $X,Y$ be $\Kum_n$-type manifolds, and choose two respective Wieneck embeddings $\iota_X, \iota_Y$. We say that $\iota_X$ and $\iota_Y$ are \textit{compatible} if for every parallel transport isometry $g\colon H^2(X,\bbZ) \to H^2(Y,\bbZ)$, there exists $\widetilde{g}\in SO(\widetilde{\Lambda})$ such that
    \begin{equation}\label{eq:compatibleembedding}
        \widetilde{g}\circ \iota_X = \iota_Y \circ g.
    \end{equation}
\end{definition}

\begin{lemma}\label{lem:compatible wieneck embeddings}
    Let $X,Y$ be $\Kum_n$-type manifolds. Then there exist compatible Wieneck embeddings $\iota_X$ and $\iota_Y$ of $X$ and $Y$ respectively.
\end{lemma}

\begin{proof}
    Fix a Wieneck embedding $\iota_Y$ and a parallel transport $g\colon H^2(X,\bbZ) \to H^2(Y,\bbZ)$ and define $\iota_X$ as $\iota_Y\circ g$. By construction, see (\ref{eqn:naturalembedding}), $\iota_X$ is a Wieneck embedding, and $g$ lifts as $\widetilde{g}=\Id\in SO(\widetilde{\Lambda})$ with respect to $\iota_X$ and $\iota_Y$. Let now $g'\colon H^2(X,\bbZ)\to H^2(Y,\bbZ)$ be another parallel transport isometry. The composition $g^{-1}\circ g'\in \Mon^2(X)$ is a monodromy operator, hence it lifts to some $\widetilde{h}\in SO(\widetilde{\Lambda})$ by Corollary~\ref{cor:Mon2isSO}, and we get that
$$\iota_Y\circ g' = \iota_Y \circ g \circ (g^{-1} \circ g')=\iota_X \circ (g^{-1}\circ g') = \widetilde{h}\circ \iota_X.$$
\end{proof}

\begin{remark}\label{rmk:liftpti}
    Let $X$ and $Y$ be two $\Kum_n$-type manifolds and fix any two Wieneck embeddings $\iota_X$ and $\iota_Y$ of $X$ and $Y$ respectively. The proof of Lemma \ref{lem:compatible wieneck embeddings} in fact reveals that:
    \begin{enumerate}
        \item For any parallel transport isometry $g\colon H^2(X,\bbZ) \to H^2(Y,\bbZ)$, there exists $\widetilde{g}\in O(\widetilde{\Lambda})$ such that $\widetilde{g}\circ\iota_X = \iota_Y\circ g$. This follows from the fact that $\iota_Y\circ g$ is a Wieneck embedding for $X$, hence differ from $\iota_X$ by postcomposition with an element of $O(\widetilde{\Lambda})$.
        \item For two parallel transport isometries $g_1,g_2\colon H^2(X,\bbZ)\to H^2(Y, \bbZ)$, we have $\det(\widetilde{g_1}) = \det(\widetilde{g_2})$ (where we define $\widetilde{g_i}$ as in part (1)). This follows from the fact that $g_1$ and $g_2$ differs by an element of $\Mon^2(Y)$ and Corollary~\ref{cor:Mon2isSO}.
    \end{enumerate}
\end{remark}

We can now prove the first main result of this paper. It can be seen as an analog of \cite[Theorem 9.8]{Marsurvey} for Kummer-type manifolds equipped with compatible Wieneck embeddings.

\begin{citedthm}\label{thm:biratU4} Let $X$ and $Y$ be $\Kum_n$-type manifolds, and let $\iota_X,\iota_Y$ be two compatible Wieneck embeddings.
Then the following are equivalent:
\begin{enumerate}
    \item[(i)] $X$ and $Y$ are bimeromorphic,
    \item[(ii)] There exists a commutative diagram
    $$\begin{tikzcd}
        H^2(X,\bbZ) \ar[r,hook,"\iota_X"] \ar[d,"g"] & \widetilde{\Lambda} \ar[d,"\widetilde{g}"] \\
        H^2(Y,\bbZ) \ar[r,hook,"\iota_Y"] & \widetilde{\Lambda}
    \end{tikzcd}$$
    such that $\widetilde{g}\in SO(\widetilde{\Lambda})$, and $g$ is a Hodge isometry.
\end{enumerate}
	\end{citedthm}

\begin{proof} 
Without loss of generality, by the proof of Lemma~\ref{lem:compatible wieneck embeddings}, we can assume that $\iota_X=\iota_Y\circ f$, where $f$ is a parallel transport isometry from $H^2(X,\bbZ)$ to $H^2(Y,\bbZ)$ (which exists since $X$ and $Y$ are deformation equivalent).

First assume that there exists a bimeromorphic map $\varphi\colon Y\to X$ and let $g:= \varphi^{\ast}\colon H^2(X,\bbZ) \to H^2(Y,\bbZ)$ be the associated Hodge parallel transport isometry (Corollary \ref{cor:Verbitskytorelli}). Then $g$ extends to $\widetilde{g}\in SO(\widetilde{\Lambda})$ by definition of compatibility of $\iota_X$ and $\iota_Y$.

Conversely, assume that there exists a commutative diagram as in the statement.
Up to changing the sign of $g$ and $\widetilde{g}$, we can assume that $g$ is orientation-preserving. By \cref{cor:Verbitskytorelli} it suffices to prove that $g$ is a parallel transport isometry, or equivalently, that $f^{-1}\circ g$ is a monodromy operator. We have $\iota_X\circ (f^{-1}\circ g)=\iota_Y\circ g = \widetilde{g}\circ \iota_X$. As $\widetilde{g}$ has determinant $1$, we get that $f^{-1}\circ g$ is a monodromy operator by Corollary~\ref{cor:Mon2isSO}.   
\end{proof}

\begin{example}[non-parallel transport Hodge isometry]\label{ex:nonMon}
The following example is taken from \cite[Section 4]{MarMeh}.
Let $A$ and $A'$ be  abelian surfaces.
   For any parallel transport isometry $f\colon H^*(A,\bbZ) \to H^*(A',\bbZ)$, the restriction $f_2\colon H^2(A,\bbZ) \to H^2(A',\bbZ)$ is admissible in the sense of \cite[Section 1]{ShiodaPeriodAbSurf}, in particular, it sends an admissible basis of $H^2(A,\bbZ)$ to an admissible basis of $H^2(A',\bbZ)$. Indeed, admissibility of a basis $u_1,u_2,u_3,u_4\in H^1(A,\bbZ)$ (and its associated basis of $H^2(A,\bbZ)=\bigwedge^2H^1(A,\bbZ)$) is defined by the property $u_1\wedge u_2\wedge u_3\wedge u_4=[pt]\in H^4(A,\bbZ)$, and this property is clearly preserved under deformation.

   Choose any marking $\eta\colon \widetilde{H}(A,\bbZ)\to \widetilde{\Lambda}$. Given an isometry $f$ as above let  $\eta'\coloneqq \eta\circ f|_{\widetilde{H}(A',\bbZ)}^{-1}$, then
 \begin{equation}\label{eq:ptiAAhat}
     \eta'\circ f|_{\widetilde{H}(A, \bbZ)} \circ \eta^{-1}=\Id_{\widetilde{\Lambda}}\in SO(\widetilde{\Lambda}).
 \end{equation}

   Now we let $A'=\hat{A}$ and denote $\hat{\eta}\coloneqq \eta'$. Shioda proved the existence of a Hodge isometry
   $$\tau_2\colon H^2(A,\bbZ) \to H^2(\hat{A},\bbZ)$$
   of determinant $-1$ (with respect to any admissible bases on both sides).
   %which is orientation-preserving. 
   Consider the extension $\tau\colon \widetilde{H}(A,\bbZ) \to \widetilde{H}(\hat{A},\bbZ)$ 
   % \yd{such that on ${H^0(A, \bbZ)\oplus H^4(A,\bbZ)}$, $\tau$ acts as $\Id$}. 
   which sends the natural generators of $H^0(A,\bbZ)$ and $H^4(A,\bbZ)$ to the corresponding ones on $\hat{A}$.
   For any positive Mukai vector $v\in \widetilde{H}(A,\bbZ)$, the image $\hat{v}\coloneqq \tau(v)$ is (up to a sign) also a positive Mukai vector, and we obtain a diagram
   $$\begin{tikzcd}
       H^2(K_A(v),\bbZ) \ar[r,hook] \ar[d,"\tau|_{v^\perp}"]& \widetilde{H}(A,\bbZ) \ar[d,"\tau"] \ar[r,"\eta"] & \widetilde{\Lambda} \ar[d,"\widetilde{\tau}"] \\
       H^2(K_{\hat{A}}(\hat{v}),\bbZ) \ar[r,hook] & \widetilde{H}(\hat{A},\bbZ) \ar[r,"\hat{\eta}"] & \widetilde{\Lambda}
   \end{tikzcd},$$
   where we define $\widetilde{\tau} \coloneqq \hat{\eta}\circ\tau\circ\eta^{-1}$.
    Remark that the embeddings obtained by composition of the horizontal maps on both lines are compatible Wieneck embeddings in the sense of Definition~\ref{def:compatible embeddings}. Indeed, this follows from the definition of $\hat{\eta}$ (see (\ref{eq:ptiAAhat})) and the fact that any parallel transport isometry $f\colon H^*(A,\bbZ) \to H^*(\hat{A},\bbZ)$ with $f(v) = \hat{v}$ restricts to a parallel transport isometry between $H^2(K_A(v),\bbZ)$ and $H^2(K_{\hat{A}}(\hat{v}),\bbZ)$. Moreover $\tau|_{v^\perp}$ is a Hodge isometry which lifts to $\widetilde{\tau}\in  O(\widetilde{\Lambda})$ with determinant $-1$. In particular, $\tau|_{v^\perp}$ is not a parallel transport isometry. Indeed, $K_v(A)$ and $K_{\hat{v}}(\hat{A})$ are not birational in general, see for instance \cite[Proposition]{Namikawa:KKdualnotbirational} (which corresponds to taking $v=(1,0,-3)=\hat{v}$).

 Finally, note that extending $\tau_2$ by reversing $H^0$ and $H^4$ of $A$ and $\hat{A}$ does give an isometry $\widetilde{\tau}$ that lifts to $SO(\widetilde{\Lambda})$, hence is a parallel transport isometry. In the light of \cref{thm:biratU4}, $K_2(A)$ is birational to $K_{\hat{v}'}(\hat{A})$ where $\hat{v}' = (2,0,-1)$, see (\ref{eqn:FMabeliansurfanddual}).

 Thus the subtelty of the sign of $\widetilde{g}$ in Theorem~\ref{thm:biratU4} is a necessary feature of the Kummer-type case.

\end{example}

\subsection{Numerical criterion for twisted modularity}
A main goal of this paper is to identify when a manifold $X$ of Kummer-type is twisted modular. For this, we use the following result, which is a Kummer-type version of \cite[Proposition~4]{AddRationalityCubic}, \cite[Lemma~2.6]{HuyK3category}. Again let $n\geq 2$, let $\Lambda_n$ be a lattice isometric to $U^{\oplus3}\oplus\langle -2-2n\rangle$ and let $\widetilde{\Lambda}$ be a lattice isometric to $U^{\oplus4}$.

\begin{proposition}\label{prop:modularityisotropic}
  Let $X$ be a manifold of $\Kum_n$-type, let $\iota_X\colon H^2(X,\bbZ)\hookrightarrow \widetilde{\Lambda}$ be a Wieneck embedding of $X$,  and equip $\widetilde{\Lambda}$ with the weight 2 Hodge structure induced by $\iota_X$ (as in Theorem~\ref{thm:PrimEmbeddingH2}). Then $X$ is twisted modular if and only if $\widetilde{\Lambda}^{1,1}\coloneqq (\widetilde{\Lambda}_\bbC)^{1,1}\cap \widetilde{\Lambda}$ contains an isotropic vector.
\end{proposition}

 \begin{proof}

Let $X$ be twisted modular, say $X=K_{A, \alpha}(v,H)$ for some abelian surface $A$ and some Brauer class $\alpha$ represented by some $B$-field $B\in H^2(A,\bbQ)$. The Mukai morphism 
\[\theta_v^{-1}\colon H^2(X,\bbZ) \simeq v^{\perp}\hookrightarrow \widetilde{H}(A, B, \bbZ),\] 
which is the Hodge isometry defined in Theorem~\ref{thm:twistedYoshioka}(2), is a Wieneck embedding of $X$ and the isotropic vector with coordinates $(0,0,1)$ lies in $\widetilde{H}^{1,1}(A, B, \bbZ)$\footnote{Note that $(1,0,0)$ is however not algebraic in general in the twisted Hodge structure $\widetilde{H}(A,B,\bbZ)$.}. This shows one direction.

For the other direction, using similar arguments as in \cite[Lemma 2.6]{HuyK3category} and Shioda's Torelli theorem \cite[Theorem II]{ShiodaPeriodAbSurf}, one can show the existence of an abelian surface $A$, a Brauer class $\alpha$ represented by some $B$-field $B\in H^2(A,\bbQ)$, a primitive positive Mukai vector $v\in \widetilde{H}(A,B,\bbZ)$ of square $v^2 = 2n+2$ and a Hodge isometry $\eta\colon \widetilde{H}(A,B,\bbZ) \rightarrow \widetilde{\Lambda}$, which restricts to a Hodge isometry $\widetilde{g}\colon v^\perp \to \iota_X(H^2(X,\bbZ))$. In particular, $v^2=\dim(X)+2\geq 6$, so we may assume that $v$ is positive in the sense of Definition $\ref{def:posvector}$. For a $v$-generic polarization $H$ on $A$, we let moreover $K_{A, \alpha}(v, H)$ be the associated $\Kum_n$-type manifold and $\theta_v\colon v^{\perp}\to H^2(K_{A, \alpha}(v, H),\bbZ)$ its Mukai morphism. 
%Up to replace $v$ by $-v$, we can assume that $v$ is positive in the sense of \Cref{def:posvector}. 
The composition $\iota\coloneqq \eta\circ \theta_v^{-1}\colon H^2(K_{A, \alpha}(v, H),\bbZ) \to \widetilde{\Lambda}$ is a Wieneck embedding of $K_{A, \alpha}(v, H)$ and one gets a diagram:

$$\begin{tikzcd}
    H^2(K_{A, \alpha}(v, H),\bbZ) \ar[r,"\iota"] \ar[d,"g","\simeq"'] & \widetilde{\Lambda} \ar[d,"\Id"] \\
    H^2(X,\bbZ) \ar[r,"\iota_X"] & \widetilde{\Lambda},
\end{tikzcd}$$
where the left vertical isometry $g$ is induced by $\widetilde{g}$.

Now, if $\iota_X$ and $\iota$ are compatible, we get that $X$ and $K_{A, \alpha}(v, H)$ are birational by Theorem~\ref{thm:biratU4}. If they are not compatible, pick $\iota_X'\coloneqq \widetilde{h}\circ \iota_X$ for some $\widetilde{h}\in O(\widetilde{\Lambda})$ with $\det(\widetilde{h})=-1$, so that $\iota_X'$ and $\iota$ are compatible. Then one gets
$$\widetilde{h}\circ \iota = \iota_X' \circ g.$$
To finish, precompose $g$ with the Hodge isometry $\tau\colon H^2(K_{\hat{A}, \hat{\alpha}}(\hat{v}, \hat{H}),\bbZ) \to H^2(K_{A, B}(v),\bbZ)$ from Example~\ref{ex:nonMon}, giving
$$(\widetilde{h}\circ  \widetilde{\tau}) \circ \hat{\iota} = \iota_X' \circ (g\circ \tau).$$
Since $\widetilde{\tau}$ has determinant $-1$, we conclude by Theorem~\ref{thm:biratU4} that $X$ is birational to $K_{\hat{A}, \hat{\alpha}}(\hat{v}, \hat{H})$.
\end{proof}

To prove Theorem~\ref{thm:main theorem}, we adapt the strategy used for \cite[Theorem 1.1]{DMPM}. Rather than treating each $\Kum_n$-type deformation family individually as $n \ge 2$ varies, we adopt a uniform approach using the lattice $\widetilde{\Lambda}$ and Wieneck embeddings. The numerical criterion established in Proposition~\ref{prop:modularityisotropic} is central to this method.

The challenge then becomes understanding how to study the symplectic birational self-maps of $\Kum_n$-type manifolds via their action on $\widetilde{\Lambda}$. Traditionally, one studies these maps on the abstract lattice $\Lambda_n$ using markings and the Hodge-theoretic Torelli-type theorems (Theorem~\ref{thm:BirmapHK}). This standard strategy involves fixing a connected component $\mathcal{M}$ of the moduli space $\mathcal{M}_{\Lambda_n}$ of $\Lambda_n$-marked pairs of $\Kum_n$-type. One then leverages the surjectivity of the period map $\mathcal{P}_{\Lambda_n}$ (Proposition~\ref{prop:surjectivity period map}), alongside descriptions of the numerical monodromy group $\Mon^2(\Lambda_n)$ and the associated numerical stably prime exceptional divisors $\mathcal{W}_{\circ}^{\rm pex}(\Lambda_n)$.

To translate this strategy to the abstract Mukai lattice $\widetilde{\Lambda}$ (Lemma~\ref{numerical torelli on mukai lattice}), we replace markings with Wieneck embeddings. This substitution allows us to relate the connected components of $\mathcal{M}_{\Lambda_n}$ to the classes of primitive embeddings $\Lambda_n \hookrightarrow \widetilde{\Lambda}$ (Lemma~\ref{lem: connected components and primitive embeddings}). Proposition~\ref{prop:Mon2isN} has already characterized the numerical monodromy operators of $\Lambda_n$ via their extensions to $\widetilde{\Lambda}$. The final necessary ingredient, a characterization in terms of $\widetilde{\Lambda}$ of the numerical stably prime exceptional divisors of $\Lambda_n$ associated to a connected component of $\mathcal{M}_{\Lambda_n}$, is provided in Lemma~\ref{lem: wall divisors on Mukai lattice} and Corollary~\ref{corollary: wall divisors via wieneck embeddings}.
\begin{lemma}\label{lem: connected components and primitive embeddings}
    Consider the map of sets
    \[\Psi_n\colon \{\text{Connected components of }\calM_{\Lambda_n}\}\to O(\widetilde{\Lambda})\backslash\{\Lambda_n\hookrightarrow\widetilde{\Lambda}\text{ primitive embedding}\},\]
    defined as follows: For a connected component $\calM_{\Lambda_n}^\circ$, the image $\Psi_n(\calM_{\Lambda_n}^\circ)$ is given by the right coset:
    \[O(\widetilde{\Lambda})\cdot( \iota_X\circ\mu^{-1}),\]
    where $[(X, \mu)]\in \calM_{\Lambda_n}^\circ$ is arbitrary and $\iota_X$ is any Wieneck embedding of $X$. Then the map $\Psi_n$ is well-defined and surjective.
\end{lemma}

\begin{proof}
    Let $\calM_{\Lambda_n}^\circ$ be a connected component of the moduli space of $\Lambda_n$-marked pairs. Fix two equivalence classes $[(X, \mu)],\,[(X', \mu')]\in \calM_{\Lambda_n}^\circ$ and let $\iota_X,\, \iota_{X'}$ be Wieneck embeddings. We already remark that the right coset $\Psi_n(\calM_{\Lambda_n}^\circ)$ does not depend on the choice of $\iota_X$ by definition. Moreover, since both classes $[(X, \mu)],\,[(X', \mu')]$ are in the same connected component, Proposition~\ref{propo:same connected component gives parallel transport} tells us that the isometry $g\coloneqq \mu^{-1}\circ\mu'\colon H^2(X', \bbZ)\to H^2(X,\bbZ)$ is a parallel transport operator. By Remark \ref{rmk:liftpti}, there exists $\widetilde{g}\in O(\widetilde{\Lambda})$ such that the following holds:
    \[\widetilde{g}\circ\iota_{X'} = \iota_X\circ g.\]
    Hence, we have an equality: $\widetilde{g}\circ\iota_{X'}\circ (\mu')^{-1} = \iota_X\circ \mu^{-1},$
    and $\Psi_n(\calM_{\Lambda_n}^\circ)$ is well-defined.

    Fix now a primitive embedding $\kappa\colon \Lambda_n\hookrightarrow\widetilde{\Lambda}$, let $(X,\mu)$ be any $\Lambda_n$-marked pair and denote $\iota\coloneqq \kappa\circ\mu$. 
    We have already seen that the Wieneck embedding $\iota_X$ is, up to the left and right actions of $O(\widetilde{\Lambda})$ and $O(H^2(X,\bbZ))$ respectively, the unique primitive embedding of $H^2(X, \mathbb{Z})$ into $\widetilde{\Lambda}$. Therefore, there exists $(\widetilde{g}, f)\in O(\widetilde{\Lambda})\times O(H^2(X,\bbZ))$ such that the following holds:
    \[\iota_X = \widetilde{g}\circ \iota\circ f.\]
    But now, if we define $\mu' := \mu\circ f$, we have that
    \[\iota_X\circ(\mu')^{-1} = \widetilde{g}\circ \kappa\]
    holds, meaning that $\kappa$ represents $\Psi_n(\calM_{\Lambda_n}^\circ)$ where $\calM_{\Lambda_n}^\circ$ is the connected component of $\mathcal{M}_{\Lambda_n}$ containing $[(X, \mu\circ f)]$.
\end{proof}

\begin{proposition}\label{lem: wall divisors on Mukai lattice}
    Let $\Psi_n$ be as in \cref{lem: connected components and primitive embeddings}.
    Let $\calM_{\Lambda_n}^\circ$ be a connected component of the moduli space of $\Lambda_n$-marked pairs, let $\kappa$ be a representative of $\Psi_n(\calM_{\Lambda_n}^\circ)$ and denote by $v\in \widetilde{\Lambda}$ any generator of $\kappa(\Lambda_n)^\perp\simeq A_1(n+1)$. Then a vector $w\in\Lambda_n$ lies in $\calW_\circ^{\rm pex}(\Lambda_n)$ if and only if there exists an isotropic vector $a\in \operatorname{Sat}_{\widetilde{\Lambda}}(\bbZ v\oplus\bbZ \kappa(w))$ such that $(a,v)\in\{1,2\}$.
\end{proposition}

\begin{proof}
    Since stably prime exceptional classes are invariant under parallel transport, (see (\ref{eq:numerical walls})), it is enough to characterize $\mathcal{W}_\circ^{\rm pex}(X)$ on $\widetilde{\Lambda}$ for any $[(X, \mu)]\in\mathcal{M}^\circ_{\Lambda_n}$ where $X = K_n(A)$ for some abelian surface $A$. 
    Let $[(X, \mu)]$ be such a class and let
    $\iota_X := \kappa\circ\mu$ be the corresponding Wieneck embedding of $X$. Note that, up to the action of $O(\widetilde{\Lambda})$, the embedding $\iota_X$ coincides with the Mukai morphism of $K_n(A)$. From this, the result follows from Yoshioka's characterization of stably prime exceptional classes on $K_n(A)$ \cite[Proposition 5.4]{Yos16} (see \cite[Theorem 12.3]{BM14} for a similar result in the case $K3^{[n]}$-cases).
\end{proof}
 
\begin{corollary}\label{corollary: wall divisors via wieneck embeddings}
    Let $X$ be a $\Kum_n$-type manifold, let $\iota_X\colon H^2(X,\bbZ)\to\widetilde{\Lambda}$ be a Wieneck embedding of $X$ and let $v\in\widetilde{\Lambda}$ be a generator of $\iota_X(H^2(X,\bbZ))^\perp$. Then a vector $w\in \NS(X)$ lies in $\mathcal{W}^{\rm pex}(X)$ if and only if there exists an isotropic vector $a\in \operatorname{Sat}_{\widetilde{\Lambda}}(\mathbb{Z}v\oplus\mathbb{Z}\iota_X(w))$ such that $(a,v)\in \{1,2\}$.
\end{corollary}

\begin{proof}
    Let $\mu\colon H^2(X, \mathbb{Z})\to v^\perp$ be the marking of $X$ induced by $\iota_X$, let $M$ be the connected component of $\mathcal{M}_{v^\perp}$ containing $[(X, \mu)]$ and let $\mathcal{W}_\circ^{\rm pex}(v^\perp)$ be the associated set of numerical stably prime exceptional divisors. By definition, we know that $\mathcal{W}^{\rm pex}(X) = \NS(X)\cap \mu^{-1}(\mathcal{W}_\circ^{\rm pex}(v^\perp))$. The result then follows directly from Proposition~\ref{lem: wall divisors on Mukai lattice}, with $\kappa$ being the inclusion $v^\perp\hookrightarrow \widetilde{\Lambda}$.
\end{proof}

\section{Symplectic birational self-maps of finite order}\label{sec:SympBirat}

\subsection{Bimeromorphic self-maps of Kummer-type manifolds}
Let again $n\geq 2$ and let $X$ be a $\Kum_n$-type manifold. The natural orthogonal representation
\begin{equation}\label{eq:natural representation}
    \rho_X\colon \Bir(X)\to O(H^2(X, \mathbb{Z}))
\end{equation}
has finite kernel \cite[Proposition 9.1]{HuycptHKbasic}, isomorphic to $\left(\ \!^{\bbZ}\!/\!_{(n+1)\bbZ}\right)^{\oplus 4}\times\ \!^{\bbZ}\!/\!_{2\bbZ}$
\cite[Corollary 5]{BNWSenriques}.
Moreover, according to \Cref{thm:BirmapHK}, we know that $\im\rho_X = \Mon^2_{\Bir}(X)$ is the group of Hodge monodromies preserving the fundamental exceptional chamber of $X$ (see Definition~\ref{def:wallchamberCX}).
A bimeromorphic self-map $f\in \Bir(X)$ is called \emph{symplectic} if its action on $H^{2,0}(X)$ is trivial. We denote by $\Bir_s(X)\subseteq \Bir(X)$ the normal subgroup consisting of symplectic bimeromorphic self-maps. 
Let again $\widetilde{\Lambda}$ be isometric to $U^{\oplus4}$.

\begin{definition}\label{def:symplectic effective}
    An isometry $g\in SO(\widetilde{\Lambda})$ is called \emph{symplectic effective} if there exists a manifold $X$ of $\Kum_n$-type, for some $n\geq 2$, and a symplectic birational self-map $f\in \Bir_s(X)$ such that $g$ is induced by $\rho_X(f)$ via some Wieneck embedding of $X$ (as in \Cref{lem:monodromyextension}).
\end{definition}

\begin{lemma}\label{numerical torelli on mukai lattice}
    Let $g\in SO(\widetilde{\Lambda})$ be a finite order isometry. Then $g$ is symplectic effective if and only if there exists a primitive vector $v\in \widetilde{\Lambda}$ with $v^2\geq 6$ such that $g(v)= \pm v$, the lattice $(\bbZ v)^\perp\cap \widetilde{\Lambda}_g$ is negative definite, where $\widetilde{\Lambda}_g$ is the coinvariant lattice, and such that there is no isotropic vector $u\in \operatorname{Sat}_{\widetilde{\Lambda}}(\widetilde{\Lambda}_g+\bbZ v)$ with $(u,v)\in\{1,2\}$.
\end{lemma}

\begin{proof}
    Let us first assume that $g$ is symplectic effective, and let $(X, f)$ be as in Definition~\ref{def:symplectic effective}. Let moreover $\iota_X$ be a Wieneck embedding of $X$ such that the following holds:
    \[\iota_X\circ \rho_X(f) = g\circ\iota_X.\]
    We denote by $v$ any generator of $\iota_X(H^2(X, \bbZ))^\perp$. It follows from Lemma~\ref{lem:monodromyextension} and its proof that $g(v) = \det(\rho_X(f))v \in\{\pm v\}$. Now the equality $(\bbZ v)^\perp\cap \widetilde{\Lambda}_g = \iota_X(H^2(X, \bbZ)_{\rho_X(f)})$ holds by construction. Since $f$ is symplectic, it fixes a nontrivial symplectic holomorphic two-form $\sigma_X$ on $X$ and the real quadratic space $H^2(X, \bbZ)^{\rho_X(f)}_\mathbb{R}$ contains a positive definite plane spanned by $\text{Re}(\sigma_X)$ and $\text{Im}(\sigma_X)$. Moreover, $\rho_X(f)$ is of finite order and it preserves the fundamental exceptional chamber $\mathcal{FE}_X$ of $X$ (see Theorem~\ref{thm:BirmapHK}), which is a convex subcone of the positive cone $\calC_X$. Therefore, by averaging the translates by $\rho_X(f)$, we see that there exists a positive vector $k_X\in \mathcal{FE}_X\subseteq \langle \text{Re}(\sigma_X), \text{Im}(\sigma_X)\rangle^{\perp}$ fixed by $\rho_X(f)$. Thus, there exists a positive definite three-space, spanned by $\text{Re}(\sigma_X)$, $\text{Im}(\sigma_X)$ and $k_X$, inside $H^2(X, \bbZ)^{\rho_X(f)}_\mathbb{R}$. It follows that $(\bbZ v)^\perp\cap \widetilde{\Lambda}_g$ is negative definite. Finally, looking for a contradiction, let us assume that there is an isotropic vector $u\in T := \operatorname{Sat}_{\widetilde{\Lambda}}(\widetilde{\Lambda}_g+\bbZ v)$ with $(u,v)\in\{1,2\}$. Together, the vectors  $v$ and $u$ spans a hyperbolic sublattice of $T$. Let  $w$ be a generator of $v^\perp\cap (\mathbb{Z}v+\mathbb{Z}u)\subseteq v^\perp\cap T\subseteq  v^\perp\cap \widetilde{\Lambda}_g$: it is a negative vector lying in the image of $H^2(X, \bbZ)_{\rho_X(f)}$ via the Wieneck embedding $\iota_X$. Since $f$ is symplectic of finite order, it follows that $H^2(X, \bbZ)_{\rho_X(f)}\subseteq \NS(X)$, meaning that $\iota_X^{-1}(w)\in \mathcal{W}^{\rm pex}(X)$ (see Corollary~\ref{corollary: wall divisors via wieneck embeddings}). But this contradicts the assumption on $k_X$: this vector is not orthogonal to any element of $\mathcal{W}^{\rm pex}(X)$ by definition of $\mathcal{FE}_X$, but it is indeed orthogonal to $\iota_X^{-1}(w)$ since $k_X$ is $\rho_X(f)$-invariant. Thus, such a primitive isotropic vector $u$ cannot exist.

    Conversely, let us assume there exists some $v\in \widetilde{\Lambda}$ with $v^2\geq 6$ such that $g(v) = \pm v$ holds, the lattice $(\bbZ v)^\perp\cap \widetilde{\Lambda}_g$ is negative definite, and such that there is no isotropic vector $u\in \operatorname{Sat}_{\widetilde{\Lambda}}(\widetilde{\Lambda}_g+\bbZ v)$ with $(u,v)\in\{1,2\}$. Let $n := \frac{v^2-2}{2}$, let $\Lambda_n := v^\perp$ and let $\kappa\colon\Lambda_n\hookrightarrow\widetilde{\Lambda}$ be the associated primitive embedding. According to Lemma~\ref{lem: connected components and primitive embeddings}, we may fix a connected component $\mathcal{M}^\circ_{\Lambda_n}$ of the moduli space of $\Lambda_n$-marked pairs such that $\Psi_n(\mathcal{M}^\circ_{\Lambda_n})$ is represented by $\kappa$.
    
    Now, since $g$ preserves $\bbZ v$, it descends to an isometry $h\in O(\Lambda_n)$ satisfying that $\det(h)h$ acts trivially on $A_{\Lambda_n}$ (see Lemma~\ref{lem:monodromyextension}). Moreover, by the assumption that $(\bbZ v)^\perp\cap \widetilde{\Lambda}_g$ is negative definite, the invariant sublattice $\Lambda_n^h$ has signature $(3, \ast)$ and $h$ preserves the orientation of the positive cone of $\Lambda_n$. Hence $h\in\Mon^2(\Lambda_n)$. Moreover, we can find a general isotropic class $\sigma\in \Lambda_n^h\otimes \bbC$ with $(\sigma,\overline{\sigma}) > 0$ such that
    \[(\bbR\Re(\sigma)\oplus \mathbb{R}\Im(\sigma))\cap \Lambda_n = \Lambda_n^h.\]
    By the surjectivity of the period map $\mathcal{P}_{\Lambda_n}$ (see Proposition~\ref{prop:surjectivity period map}), there exists a $\Lambda_n$-marked pair $(X, \mu)$ with $[(X, \mu)]\in \mathcal{M}^\circ_{\Lambda_n}$ and such that $\mu(H^{2,0}(X)) = \bbC\sigma.$ It satisfies that $\mu(\NS(X)) = (\Lambda_n)_h = (\bbZ v)^\perp\cap \widetilde{\Lambda}_g$. The isometry $h_X \coloneqq\mu^{-1}h\mu$ lies in $\Mon^2(X)$ because $h\in \Mon^2(\Lambda_n)$ holds. Furthermore, $\NS(X) = \mu^{-1}((\Lambda_n)_h)$ contains no elements of $\mathcal{W}^{\rm pex}(X)$ by our assumption on $(\bbZ v)^\perp\cap \widetilde{\Lambda}_g$, by Corollary~\ref{corollary: wall divisors via wieneck embeddings} and the fact that $\kappa\circ\mu$ is a Wieneck embedding of $X$ (see also the first part of the proof). Thus $h_X$ preserves an exceptional chamber of positive cone of $X$, and up to replacing $X$ by a bimeromorphic model (allowed by Theorem~\ref{theo:semidirect decomposition}), we may assume that $h_X$ preserves $\mathcal{FE}_X$. But now, since $h_X$ is Hodge by construction, it lies in $\Mon^2_{\rm Bir}(X)$ and it is of the form $\rho_X(f)$ for some $f\in \Bir(X)$. Note that such an $f$ is symplectic since $h_X$ acts trivially on $H^{2,0}(X)$ by construction. Hence, $g$ is symplectic effective.
\end{proof}

\begin{remark}\label{rem adapt the proof to the projective case}
    In the proof of \Cref{numerical torelli on mukai lattice}, we realize the isometry $g$ by representing on $\widetilde{\Lambda}$ a symplectic bimeromorphic self-map of some nonprojective $\Kum_n$-type manifold $X$ such that $\NS(X)$ is identified to $v^\perp\cap \widetilde{\Lambda}_g$ via some Wieneck embedding of $X$. Note that with the assumption of the statement of \Cref{numerical torelli on mukai lattice}, we have that $v^\perp\cap\widetilde{\Lambda}^g$ has signature $(3,\ast)$. Therefore one may adapt the proof by fixing some primitive vector $h\in v^\perp\cap\widetilde{\Lambda}^g$ of positive square and pick, this time, a general period point in $\mathbb{P}((h^\perp\cap v^\perp\cap \widetilde{\Lambda}^g)\otimes \mathbb{C})$. From this, we obtain a manifold $X$ as in the proof which is projective with:
    \[\NS(X) \simeq \operatorname{Sat}_{v^\perp}(\mathbb{Z}h\oplus \widetilde{\Lambda}_g).\]
     So for any $d\in \bbN$ such that $v^\perp\cap\widetilde{\Lambda}^g$ represents $2d$ primitively, we may actually realize $g$ on some projective Kummer-type manifold equipped with a (quasi)polarization of BBF-square $2d$.
\end{remark}

\subsection{Non-modular symmetric examples}
In the paper \cite{DMPM}, the first two named authors and Prieto-Monta\~nez show that any projective $\text{K3}^{[n]}$-type hyperk\"ahler manifold $X$ admitting a symplectic birational self-map $f$ of finite order is twisted modular. One of the main steps of the proof is the following: if $f$ acts trivially on the discriminant group of $H^2(X, \mathbb{Z})$, then the Picard rank of $X$ is at least 9 and a theorem of Meyer gives that $\NS(X)$, which has signature $(1,\ast)$, contains isotropic vectors. Then they concluded modularity of $X$ using the numerical criterion, i.e., a criterion similar to \cref{prop:modularityisotropic} valid for $K3^{[n]}$-type manifolds. They show more generally that the coinvariant lattice of the isometry induced by $f$ on the associated Mukai lattice $U^{\oplus4}\oplus E_8^{\oplus2}(-1)$ is isotropic.

This argument does not carry over directly to the case of Kummer-type manifolds, primarily because the numerical classifications of stably prime exceptional divisors differ between the two settings. Notably, in the case of Kummer-type manifolds, classes of square $-2$ do not represent prime exceptional divisors. This raises a crucial difference between the two cases: if a finite order isometry of an even unimodular lattice (of rank at most 24) has its coinvariant lattice $C$ without $(-2)$-vectors, then the rank of $C$ is at least 8 \cite[Proposition 1.2]{Zhengleech}, \cite[Table 1]{zbMATH06530862}. This difference was previously already observed by Mongardi, Tari and Wandel who determined the possible invariant and coinvariant sublattices for some symplectic effective prime order isometries of $\widetilde{\Lambda} \coloneqq U^{\oplus4}$ \cite[Sections 2, 5 and 6]{MTWAutKummerFourfold}. They classify two types of isometries, of order 2 and 3 respectively, with coinvariant sublattice of small rank. In Lemma~\ref{lem: special isometries mukai lattice}, we give an explicit description for such isometries. 

\begin{lemma}\label{lem: special isometries mukai lattice}
    There exist two isometries $g_2,g_3\in SO(\widetilde{\Lambda})$ of respective order 2 and 3 satisfying the following:
    \begin{enumerate}
        \item both $g_2$ and $g_3$ fixes pointwise two orthogonal copies of $U$,
        \item $\widetilde{\Lambda}_{g_2}\simeq A_1^{\oplus2}(-1)$
        and $\widetilde{\Lambda}_{g_3}\simeq A_2(-1)$.
    \end{enumerate}
\end{lemma}

\begin{proof}
    Write $\widetilde{\Lambda}$ as $U_1\oplus U_2\oplus U_3\oplus U_4$ where for each $i\in \{1,\ldots, 4\}$, we have $U_i\simeq U$. For each such index $i$, we let $e_i, f_i\in U_i$ be isotropic vectors such that $(e_i,f_i) = 1$. Then the result follows directly by defining $g_2, g_3$ being the identity on $U_1\oplus U_2$ and acting on the basis $\{e_3,f_3,e_4,f_4\}$ of $U_3\oplus U_4$ respectively as
    \[
        (g_2)_{\mid U_3\oplus U_4}= \begin{pmatrix}
            0&1&0&0\\1&0&0&0\\0&0&0&1\\0&0&1&0
        \end{pmatrix}\qquad
        (g_3)_{\mid U_3\oplus U_4} = \begin{pmatrix}
            0&1&0&1\\1&0&0&-1\\0&0&0&1\\-1&1&1&1
        \end{pmatrix}.\]
\end{proof}

\begin{remark}
    Note that $g_2$ (resp. $g_3$) is, up to conjugacy in $O(\widetilde{\Lambda})$, the unique isometry of $\widetilde{\Lambda}$ whose coinvariant sublattice is isometric to $A_1^{\oplus2}(-1)$ (resp. $A_2(-1)$) (it follows by a straightforward combination of \cite[Corollary 1.5.2, Corollary 1.6.2, Theorem 1.10.1 and Theorem 1.14.2]{nik79}).
\end{remark}

We show in Proposition~\ref{prop:exceptions} and Lemma~\ref{lem: nonmodularity of exceptions} that the existence of the isometries $g_2$ and $g_3$ from Lemma~\ref{lem: special isometries mukai lattice} obstructs full generalization of \cite[Theorem 1.1]{DMPM} to the case of Kummer-type manifolds. In order to do so, we first prove the following standard results about isometries of (even) lattices.

\begin{lemma}\label{lem:possible char poly}
    Let $L$ be a free $\bbZ$-module of finite rank $r\in\{2,3\}$ and let $f\in \GL(L)$ be an automorphism of finite order $n>1$ that fixes a nonzero vector in $L$. Let us denote by $\chi(t)\in\bbZ[t]$ the characteristic polynomial of $f$.
    Then, exactly one of the following holds:
    \begin{enumerate}
        \item either $n=2$, $\chi(t) = (t-1)^{1+a}(t+1)^{r-1-a}$ with $a\in\{0,r-2\}$ and $\det(f) = (-1)^{r-1-a}$,
        \item or $r = 3$, $n =3,4,6$, $\chi(t) = (t-1)\Phi_n(t)$ and $\det(f) = 1$.
    \end{enumerate}
\end{lemma}

\begin{proof}
    Since $f$ is of finite order, the characteristic polynomial of $f$ is a product of cyclotomic polynomials. Since $f$ fixes a nonzero vector in $L$, we already know that $\chi(1) = 0$. The proof follows by writing all possible characteristic polynomials for $f$, depending on the rank of $L$. 
\end{proof}

\begin{lemma}\label{lem:action on discriminant}
    Let $L$ be an even lattice of signature $(l_+, l_-)$, and let $f\in O(L)$ be an isometry of finite order $n>1$ satisfying both of the following:
    \begin{enumerate}
        \item[(a)] $f$ acts trivially on the discriminant group $A_L$ of $L$,
        \item[(b)] $\Phi_n(f) = 0$.
    \end{enumerate}
    Then, $\rk(L)$ is divisible by the Euler totient $\varphi(n)$ and exactly one of the following holds:
    \begin{enumerate}
        \item $n = 2^k$ for some $k\geq 1$, $L$ is $2$-elementary and $l_+-l_-$ is divisible by $4$ if $n\geq 4$;
        \item $n = p^k$ for $p$ an odd prime number and some $k\geq 1$, and $L$ is $p$-elementary;
        \item $n$ is composite, $L$ is unimodular and $l_+-l_-$ is divisible by 8.
    \end{enumerate}
\end{lemma}

\begin{proof}
    The characteristic polynomial of $f$ is a power of $\Phi_n$, so the rank of $L$ is divisible by the degree of $\Phi_n$, which is exactly $\varphi(n)$. Now, by assumption, $\Phi_n(f)$ acts by multiplication by $0$ on $A_L$. But since $f$ induces the identity on $A_L$, it follows that $\Phi_n(1)$ annihilates $A_L$. The value $\Phi_n(1)$ is $p$ if $n$ is a nontrivial power of the prime number $p$, and 1 for $n\geq 2$ composite. Hence $A_L$ is either an elementary abelian $p$-group, or it is trivial, respectively. Moreover, if $n = 2^k$ for some $k\geq 2$, we know that $L$ is 2-elementary and \cite[Proposition 2.4]{taki12} tells us that the quadratic form on $L^\vee$ takes integer values, which implies that $l_+-l_-$ is divisible by 4 \cite[Chapter 15, Theorem 13]{splg}. Finally the last claim follows from \cite[Chapter V, Section 2, Theorem 2 \& Corollary 1]{Serre}.
\end{proof}

\begin{proposition}\label{prop:exceptions}
    Let $X$ be a projective $\Kum_n$-type manifold of Picard rank 3, for some $n\geq 2$, and suppose that there exists a symplectic birational self-map $f\in \Bir(X)$ of finite order such that $f^\ast\in \Mon^2(X)$ is nontrivial and has determinant $+1$. Then there exists a positive integer $d\geq 1$ such that exactly one the following two holds:
    \begin{enumerate}
        \item[\textnormal{(I)}] $\ord(f^*) = 2$ and there is a finite index embedding $\langle 2d\rangle\oplus  A_1^{\oplus2}(-1)\hookrightarrow\NS(X)$,
        \item[\textnormal{(II)}] $\ord(f^*) = 3$ and there is a finite index embedding $\langle 2d\rangle\oplus  A_2(-1)\hookrightarrow\NS(X)$.
    \end{enumerate}
    Moreover, for every $n\geq 2$ and $d\geq 1$, a pair $(X, f)$ satisfying the above conditions exists.
\end{proposition}

\begin{proof}
    Suppose that $X$ and $f$ as in the statement exists, let $g$ be the restriction of $f^\ast$ to $\NS(X)$ and let $m\geq 2$ denote the order of $g$. Note that since $f$ is symplectic, the isometry $f^\ast$ acts trivially on the the transcendental lattice $\NS(X)^\perp$ of $X$ and therefore
    $\ord(g) = \ord(f^\ast)$. Moreover, since $\det(f^\ast) = +1$, by Proposition \ref{prop:Mon2isN} it follows that the isometry $f^\ast$ acts trivially on the discriminant group of $H^2(X, \bbZ)$. Furthermore, it follows that $g$ acts trivially on $A_{\NS(X)}$.

    Now, since $X$ is projective and $f$ is of finite order, the isometry $f^\ast$ fixes a class $h\in \NS(X)$ of positive square $2d$, for some $d\geq 1$, and it is clear that $h$ is fixed by $g$ as well by definition (choose $h$ being the sum of the elements in the orbit of an ample class under the action of $f$ on $\NS(X)$). Therefore, since $\rk\NS(X) = 3$, \Cref{lem:possible char poly} tells us that $m = 2,3,4,6$, the coinvariant lattice $\NS(X)_g$ is even of rank 2, and the minimal polynomial of the restriction of $g$ to $\NS(X)_g$ is $\Phi_m$. But since $g$ acts trivially on the discriminant group of $\NS(X)$, we obtain that $g$ acts trivially on the discriminant group of $\NS(X)_g$. Thus, according to \Cref{lem:action on discriminant}, we must have $m=2,3$. In the case where $m = 2$ we have that $\NS(X)_g$ is a negative definite 2-elementary lattice of rank 2, so it is isometric to $A_1^{\oplus2}(-1)$. In the case where $m=3$, the lattice $\NS(X)_g$ is negative definite, 3-elementary and of rank 2, so it has to be isometric to $A_2(-1)$. The claim follows.\smallskip

    For the existence statement, we prove only the type \textnormal{(I)} case: the other case follows similarly.
    Let $g_2\in SO(\widetilde{\Lambda})$ be as in \Cref{lem: special isometries mukai lattice}. Then, by its definition, we have that
    \[\widetilde{\Lambda}^{g_2}\simeq U_1\oplus U_2\oplus A_1^{\oplus2}\qquad\text{and}\qquad \widetilde{\Lambda}_{g_2}\simeq A_1^{\oplus2}(-1),\]
    where for $i=1,2$, we have $U_i\simeq U$. Again, for both $i=1,2$, we denote by $e_i,f_i$ isotropic vectors of $U_i$ such that $(e_i,f_i)=1$. Let now $h := de_1+f_1\in U_1$ and $v = (n+1)e_2+f_2\in U_2$: these vectors are primitive in $\widetilde{\Lambda}$ and they are orthogonal to each other. Moreover, by construction, we have the following equalities:
    \[\operatorname{Sat}_{\widetilde{\Lambda}}(\mathbb{Z}v\oplus \widetilde{\Lambda}_{g_2}) = \mathbb{Z}v\oplus \widetilde{\Lambda}_{g_2}\qquad\text{and}\qquad \operatorname{Sat}_{\widetilde{\Lambda}}(\mathbb{Z}h\oplus \widetilde{\Lambda}_{g_2}) = \mathbb{Z}h\oplus \widetilde{\Lambda}_{g_2}.\]
    In particular, if $a\in \operatorname{Sat}_{\widetilde{\Lambda}}(\mathbb{Z}v\oplus \widetilde{\Lambda}_{g_2})$ is isotropic, the product $(a,v)$ is a nonzero positive multiple of $v^2 = 2n+2\geq 6$. Hence, according to Lemma~\ref{numerical torelli on mukai lattice} and Remark~\ref{rem adapt the proof to the projective case}, there exists a pair $(X, f)$ where $X$ is a $\Kum_n$-type manifold with $\NS(X)\simeq \operatorname{Sat}_{\widetilde{\Lambda}}(\mathbb{Z}h\oplus \widetilde{\Lambda}_{g_2})$ and $f\in \Bir(X)$ is a finite order symplectic birational self-map such that $f^\ast$ is an involution fixing a class of square $2d$ in $\NS(X)$ and whose coinvariant lattice is isometric to $A_1^{\oplus2}(-1)$. Note that $f^\ast$ has determinant $+1$ since its characteristic polynomial is equal to $(t-1)^5(t+1)^2\in \bbQ[t]$.
\end{proof}

\begin{lemma}\label{lem: nonmodularity of exceptions}
    Let $n\geq 2$ and $d\geq 1$ be positive integers and let $X$ be a $\Kum_n$-type manifold such that $\NS(X)$ contains a finite index sublattice isometric to either one of the following two:
    \begin{enumerate}
        \item[\textnormal{(I)}] $\langle2d\rangle\oplus A_1^{\oplus2}(-1)$,
        \item[\textnormal{(II)}] $\langle2d\rangle\oplus A_2(-1)$.
    \end{enumerate}
    Then $X$ is twisted modular if and only if
    \begin{enumerate}
        \item[\textnormal{(I)}] the quaternary quadratic form $ (n+1)x^2+dy^2-z^2-t^2$ is isotropic,
        \item[\textnormal{(II)}] the quaternary quadratic form $ (n+1)x^2+dy^2-3z^2-t^2$ is isotropic.
    \end{enumerate}
\end{lemma}

\begin{proof}
    By Proposition~\ref{prop:modularityisotropic}, in order to determine whether $X$ is twisted modular, it is equivalent to determine whether the extended N\'eron--Severi lattice $\widetilde{\Lambda}^{1,1}$ is isotropic. Since being isotropic is a rational property, it is therefore equivalent to determining when $\widetilde{\Lambda}^{1,1}_\bbQ$ is isotropic. Since there is a finite index embedding $\langle2n+2\rangle\oplus \NS(X)\hookrightarrow\widetilde{\Lambda}^{1,1}$, we obtain the following isometric rational quadratic spaces:
    \begin{enumerate}
        \item[\textnormal{(I)}] $\widetilde{\Lambda}^{1,1}_\bbQ\simeq \langle 2n+2, 2d, -2, -2\rangle_\bbQ$,
        \item[\textnormal{(II)}] $\widetilde{\Lambda}^{1,1}_\bbQ\simeq \langle 2n+2, 2d, -6, -2\rangle_\bbQ$.
    \end{enumerate}
    For the type (II), we use the fact that there exists an embedding $\langle-6, -2\rangle\hookrightarrow A_2(-1)$ of finite index.
    Up to rescaling the previous quadratic forms by $\frac{1}{2}$, the result follows.
\end{proof}

\begin{remark}
    Testing whether an indefinite rational quadratic form is isotropic is a standard routine, and explicit numerical criteria are known \cite[Chapter IV, Theorems 6 and 8]{Serre}. In particular, one can easily conclude that there are infinitely many pairs $(n, d)\geq (2,1)$ such that the quadratic forms from \Cref{lem: nonmodularity of exceptions} are (an)isotropic. In practice, one would check for isotropy using some computer algebra system featuring functionalities about number theory and rational quadratic forms. For the reader's convenience, we use the software Hecke \cite{nemo} to determine all the pairs $(n,d)$ with $n\in\{2,\ldots, 20\}$ and $d\in \{1,\ldots, 100\}$ for which the quadratic forms in \Cref{lem: nonmodularity of exceptions} are anisotropic. We report these values in \Cref{tab:tab1,tab:tab2}.
\end{remark}

\begin{table}[ht]
    \centering
    \renewcommand\arraystretch{1.2}
    \begin{tabular}{c|l}
         $n$&$d$\\
         \hline
         \rowcolor{gray!30!white}2& $3, 11, 12, 19, 21,27,30,35,39,43,44,48,51,57,59,66,67,75,76,83, 84,91,93,99$\\
         5& $6,15,22,24,33,38,42,51,54,60,69,70,78,86,87,88,96$\\
         \rowcolor{gray!30!white}6& $7, 14, 15, 23, 28, 31, 39, 47, 55, 56, 60, 63, 71, 77, 79, 87, 92, 95$\\
         10& $3, 11, 12, 19, 27, 33, 35, 43, 44, 48, 51, 55, 59, 67, 75, 76, 83, 91, 99$\\
         \rowcolor{gray!30!white}11&$3, 11, 12, 19, 21, 27, 30, 35, 39, 43, 44, 48, 51, 57, 59, 66, 67, 75, 76, 83, 84, 91, 93, 99$\\
         13&$7, 14, 28, 30, 46, 56, 62, 63, 77, 78, 94$\\
         \rowcolor{gray!30!white}14&$6, 7, 15, 23, 24, 28, 31, 33, 39, 42, 47, 51, 54, 55, 60, 63, 69, 71, 78, 79, 87, 92, 95, 96$\\
         18&$3, 11, 12, 19, 27, 35, 43, 44, 48, 51, 59, 67, 75, 76, 83, 91, 95, 99$\\
         \rowcolor{gray!30!white}20&$3, 12, 21, 27, 30, 35, 39, 42, 48, 57, 66, 70, 75, 84, 91, 93$
    \end{tabular}
    \caption{Pairs $(n,d)\leq (20,100)$ with $(n+1)x^2+dy^2-z^2-t^2$ anisotropic}\label{tab:tab1}
\end{table}

\begin{table}[ht]
    \centering
    \renewcommand\arraystretch{1.2}
    \begin{tabular}{c|l}
         $n$&$d$\\
         \hline
         \rowcolor{gray!30!white}4& $6, 10, 15, 24, 33, 35, 40, 42, 51, 54, 60, 65, 69, 78, 85, 87, 90, 96$\\
         5& $2, 5, 8, 11, 14, 17, 18, 20, 23, 26, 29, 32, 34, 35, 38, 41, 44, 45, 47, 50, 53, 56, 59, 62, 65,$\\
         &$ 66, 68, 71, 72, 74, 77, 80, 82, 83, 86, 89, 92, 95, 98, 99$\\
         \rowcolor{gray!30!white}7& $6, 15, 22, 24, 33, 38, 42, 51, 54, 60, 69, 70, 78, 86, 87, 88, 96$\\
         9&$5, 14, 20, 30, 45, 46, 55, 56, 62, 70, 78, 80, 94, 95$\\
         \rowcolor{gray!30!white}10&$6, 11, 15, 24, 33, 42, 44, 51, 54, 55, 60, 69, 78, 87, 96, 99$\\
         13&$6, 10, 15, 24, 26, 33, 40, 42, 51, 54, 58, 60, 69, 74, 78, 87, 90, 96$\\
         \rowcolor{gray!30!white}14&$2, 5, 8, 11, 14, 17, 18, 20, 23, 26, 29, 30, 32, 35, 38, 41, 44, 45, 47, 50, 53, 55, 56, 59, 62,$\\
         \rowcolor{gray!30!white}&$ 65, 68, 70, 71, 72, 74, 77, 80, 83, 86, 89, 92, 95, 98, 99$\\
         16&$36, 15, 24, 33, 42, 51, 54, 60, 69, 78, 85, 87, 96$\\
         \rowcolor{gray!30!white}17&$6, 15, 22, 24, 33, 38, 42, 51, 54, 60, 69, 70, 78, 86, 87, 88, 96$\\
        19&$6, 10, 15, 24, 33, 35, 40, 42, 51, 54, 60, 65, 69, 78, 85, 87, 90, 96$
    \end{tabular}
    \caption{Pairs $(n,d)\leq (20,100)$ with $(n+1)x^2+dy^2-3z^2-t^2$ anisotropic}\label{tab:tab2}
\end{table}

\subsection{Proof of the main theorem}

In this section, we prove 
%the main theorem of this work. It states 
that apart from the exceptions described Proposition~\ref{prop:exceptions} and Lemma~\ref{lem: nonmodularity of exceptions}, the existence of a finite order symplectic birational self-map of a projective manifold $X$ of Kummer-type is enough to ensure the twisted modularity of $X$ (see Definition~\ref{def:modularKumn}). 
The proof, however, is different from the $K3^{[n]}$ case (see \cite[Theorem 1.1]{DMPM}) since one of the key inputs there (see [Lemma 2.9, \textit{loc.\ cit.}]) involves numerical conditions for extremal wall divisors. 
As these are different for the case of Kummer-type manifolds (they do not include $(-2)$-roots) the proof in \textit{loc.\ cit} does not apply to our current setup.

\begin{theorem}\label{thm:modular except for some examples}
    Let $X$ be a projective $\Kum_n$-type manifold, for some $n\geq2$. 
    Assume that there exists a symplectic birational self-map $f\in \Bir(X)$ such that $f^*\in O(H^2(X,\bbZ))$ has finite order greater than $1$.
   Then $X$ is twisted modular if and only if there does not exist any positive integer $d\geq 1$ such that either of the following two holds:
   \begin{enumerate}
       \item[\textnormal{(I)}] There is a finite index embedding $\langle2d\rangle\oplus A_1^{\oplus2}(-1)\hookrightarrow \NS(X)$ and $(n+1)x^2+dy^2-z^2-t^2$ is anisotropic
       \item[\textnormal{(II)}] There is a finite index embedding $\langle2d\rangle\oplus A_2(-1)\hookrightarrow \NS(X)$ and $(n+1)x^2+dy^2-3z^2-t^2$ is anisotropic.
   \end{enumerate}
\end{theorem}

\begin{proof}
   By Proposition~\ref{prop:modularityisotropic}, it is enough to determine when the extended N\'eron--Severi lattice $\widetilde{\Lambda}^{1,1}$ admits an isotropic vector. We fix a Wieneck embedding $\iota\colon H^2(X, \mathbb{Z})\hookrightarrow\widetilde{\Lambda}\coloneqq U^{\oplus4}$, we let $v$ be an associated Mukai vector, i.e. such that $\bbZ v = \iota(H^2(X, \mathbb{Z}))^\perp$, and we define 
   \[\widetilde{\Lambda}^{1,1} := \operatorname{Sat}_{\widetilde{\Lambda}}(\iota(\NS(X))\oplus\bbZ v).\]
   Let us denote $g := f^\ast$ and let $\widetilde{g}$ be the extension to $\widetilde{\Lambda}$ (see the notation from \Cref{cor:Mon2isSO}). Note that $g$ and $\widetilde{g}$ have the same order. 
    Since $f$ is symplectic, $g$ fixes pointwise the transcendental lattice of $X$ and we have the equality of orders $\ord(g) = \ord(g|_{\NS(X)})$ (meaning in particular that $\rk\NS(X) \geq 2$). Moreover, by the description of the monodromy of $X$, we know that $\det(g) = -1$ if and only if $\widetilde{\Lambda}^{\widetilde{g}} = \iota(H^2(X,\bbZ)^g)$ holds, and $\det(g) = +1$ if and only if  $\widetilde{\Lambda}_{\widetilde{g}} = \iota(H^2(X,\bbZ)_g)$ holds. In particular, since $H^2(X,\bbZ)_g\subseteq \NS(X)$ is negative definite, because $f$ is symplectic, we obtain that $\widetilde{\Lambda}_{\widetilde{g}}\subseteq \widetilde{\Lambda}^{1,1}$ is negative definite whenever $\det(g) = +1$, and hyperbolic otherwise.

    When $\rk \NS(X) \geq 4$, we have $\rk \widetilde{\Lambda}^{1,1} \geq 5$ and $\widetilde{\Lambda}^{1,1}$ is isotropic according to Meyer's theorem. Therefore, from now on, we assume that $2\leq \rk \NS(X) \leq 3$.

    Let $m\geq 2$ be the order of $g$. Since $g$ has finite order and it preserves the movable cone 
    $\Mov(X)$, which is a convex cone, it fixes a nonzero vector in $\Mov(X)$: namely for any ample divisor $D\in \NS(X)$, the following holds: $0\neq w\coloneqq (D+\cdots + (g^{m-1})(D))\in \Mov(X)^g$.

    \textbf{Case 1.} Assume $\rk\NS(X)=2$. Then according to \Cref{lem:possible char poly}, $g$ is a reflection and the coinvariant lattice of the involution $\widetilde{g}$ is thus hyperbolic. Moreover, since $\widetilde{\Lambda}$ is unimodular, $\widetilde{g}$ acts trivially on the discriminant group of $\widetilde{\Lambda}_{\widetilde{g}}$, which is therefore 2-elementary (\Cref{lem:action on discriminant}).
    In particular, according to \cite[Theorem 4.3.3]{NikulinQuotient}, $\widetilde{\Lambda}_{\widetilde{g}}\subseteq \widetilde{\Lambda}^{1,1}$ contains an isotropic vector.
    
    \textbf{Case 2.} Assume $\rk\NS(X)=3$ and $\det(g)=-1$. According to \Cref{lem:possible char poly} again, we must have $\ord(g)=2$, and the result follows similarly as in \textbf{Case 1}.

    \textbf{Case 3.} If $\rk\NS(X)=3$ and $\det(g) = +1$, then we can apply \Cref{prop:exceptions,lem: nonmodularity of exceptions} to give exactly which cases are not modular, concluding the proof.
 \end{proof}

We conclude by giving a complete characterization of the non-modular cases from Theorem~\ref{thm:modular except for some examples} via their N\'eron--Severi lattices.

\begin{corollary}\label{cor: NeronSeveriExceptions}
    Let $X$ be a projective $\Kum_n$-type manifold, for some $n\geq2$, and suppose that there exists a symplectic birational self-map $f\in \Bir(X)$ such that $f^*\in O(H^2(X,\bbZ))$ has finite order greater than $1$. Then $X$ is not twisted modular if and only if there exists a positive integer $d\geq 1$ such that one of the following holds:
    \begin{enumerate}
        \item the form $(n+1)x^2+dy^2-z^2-t^2$ is anisotropic and there is a basis $B$ of $\NS(X)$ such that
        \[\NS(X) = \begin{pmatrix}
            2d&0&0\\0&-2&0\\0&0&-2
        \end{pmatrix}_B\qquad \text{and}\qquad f^\ast|_{\NS(X)}=\begin{pmatrix}
            1&0&0\\0&-1&0\\0&0&-1
        \end{pmatrix}_B;\]
        \item the form $(n+1)x^2+(4d+1)y^2-z^2-t^2$ is anisotropic and there is a basis $B$ of $\NS(X)$ such that
        \[\NS(X) = \begin{pmatrix}
            2d&-1&0\\-1&-2&0\\0&0&-2
        \end{pmatrix}_B\qquad \text{and}\qquad f^\ast|_{\NS(X)}=\begin{pmatrix}
            1&0&0\\-1&-1&0\\0&0&-1
        \end{pmatrix}_B;\]
        \item the form $(n+1)x^2+(4d+2)y^2-z^2-t^2$, is anisotropic and there is a basis $B$ of $\NS(X)$ such that
        \[\NS(X) = \begin{pmatrix}
            2d&-1&-1\\-1&-2&0\\-1&0&-2
        \end{pmatrix}_B\qquad \text{and}\qquad f^\ast|_{\NS(X)}=\begin{pmatrix}
            1&0&0\\-1&-1&0\\-1&0&-1
        \end{pmatrix}_B;\]
        \item the form $(n+1)x^2+dy^2-3z^2-t^2$ is anisotropic and there is a basis $B$ of $\NS(X)$ such that
        \[\NS(X) = \begin{pmatrix}
            2d&0&0\\0&-2&-1\\0&-1&-2
        \end{pmatrix}_B\qquad \text{and}\qquad f^\ast|_{\NS(X)}=\begin{pmatrix}
            1&0&0\\0&0&1\\0&-1&-1
        \end{pmatrix}_B;\]
        \item the form $(n+1)x^2+(9d+3)y^2-3z^2-t^2$ is anisotropic and there is a basis $B$ of $\NS(X)$ such that
        \[\NS(X) = \begin{pmatrix}
            2d&-1&-1\\-1&-2&-1\\-1&-1&-2
        \end{pmatrix}_B\qquad \text{and}\qquad f^\ast|_{\NS(X)}=\begin{pmatrix}
            1&0&0\\0&0&1\\-1&-1&-1
        \end{pmatrix}_B.\]
    \end{enumerate}
    Moreover, all the cases (1)-(5) can occur.
\end{corollary}

\begin{proof}
    Let $X$ and $f$ be as in the first part of the statement and assume that $f^\ast$ has order 2. According to Theorem~\ref{thm:modular except for some examples} and its proof, we know that $X$ is not twisted modular if and only if there exists $d\geq 1$ such that $\NS(X)$ is isometric to an even overlattice of $\langle2d\rangle\oplus A_1^{\oplus2}(-1)$ and $(n+1)x^2+dy^2-z^2-t^2$ is anisotropic. The first condition is equivalent to saying that there exist pairwise orthogonal vectors $\alpha,\beta,\gamma\in \NS(X)$ with $\alpha^2 = 2d$ and $\beta^2 = \gamma^2 = -2$, and such that we have an equality of rational quadratic forms:
    \[\NS(X)_\mathbb{Q} = \Span_\bbQ(\alpha,\beta,\gamma).\]
    Now, the sublattice $A_1^{\oplus2}(-1)\simeq \bbZ\beta\oplus\bbZ\gamma\subset \NS(X)$ admits no nontrivial even overlattice \cite[Proposition 1.4.1]{nik79} meaning that it is primitive in $\NS(X)$. Moreover, there exists a unique positive integer $k\geq 1$ such that the vector $\alpha$ is primitive in $k\NS(X)$. Note that $(n+1)x^2+dy^2-z^2-t^2$ is anisotropic if and only if so is $(n+1)x^2+\frac{d}{k^2}y^2-z^2-t^2$. So up to replacing $\alpha$ by $\alpha/k$ and $d$ by $d/k^2$, we may assume that $\alpha$ is primitive in $\NS(X)$. The upshot is the following: The lattice $\NS(X)$ is a primitive extension of $\bbZ\alpha\oplus (\bbZ\beta\oplus\bbZ\gamma)$ (i.e. it is an overlattice and both of the summands $\bbZ\alpha$ and $\bbZ\beta\oplus\bbZ\gamma$ are primitive in $\NS(X)$). Following Nikulin's theory on primitive extensions of even lattices (see for instance \cite[Section 1, 5°]{nik79}), up to permuting $\beta$ and $\gamma$, exactly one of the following three holds:
    \begin{enumerate}
        \item either $\NS(X) = \Span_\bbZ(\alpha,\beta,\gamma)$,
        \item or $d\equiv 1\mod 4$ and $\NS(X) = \Span_\bbZ\left(\frac{\alpha+\beta}{2},\beta,\gamma\right)$,
        \item or $d\equiv 2\mod 4$ and $\NS(X) = \Span_\bbZ\left(\frac{\alpha+\beta+\gamma}{2},\beta,\gamma\right)$,
    \end{enumerate}
    giving cases (1)--(3), respectively.
    For cases (4) and (5) one proceeds similarly, by assuming that the isometry $f^\ast$ has order 3.  The details are left to the reader.\smallskip

    Now let us show that the cases (1)--(5) actually occur. Let $\widetilde{\Lambda}\coloneqq U_1\oplus U_2\oplus U_3\oplus U_4$ where for each $i\in \{1,\ldots, 4\}$, we have $U_i\simeq U$. For each such index $i$, we let again $e_i, f_i\in U_i$ be isotropic vectors such that $(e_i,f_i) = 1$. We let moreover $g_2,g_3\in SO(\widetilde{\Lambda})$ denote the isometries defined as in the proof of Lemma~\ref{lem: special isometries mukai lattice}. In \Cref{prop:exceptions,lem: nonmodularity of exceptions}, we show that cases (1) and (4) from the statement can occur. For this, for fixed $n$ and $d$, we define $v \coloneqq (n+1)e_2+f_2$ and $h \coloneqq de_1+f_1$, and we show the existence of a  $(v^\perp)$-marked pair $(X, \mu)$ where $X$ is a projective $\Kum_n$-type manifold such that $\mu(\NS(X)) = \bbZ h\oplus \widetilde{\Lambda}_{g_i}$ for some $i=1,2$. To realize the other cases, we proceed similarly but we change the $g_i$-invariant primitive positive vector $h$ we construct.

    By definition of $g_2$, the following equalities hold: $\widetilde{\Lambda}^{g_2} = \Span_\bbZ(e_1,f_1,e_2,f_2,e_3+f_3,e_4+f_4)$ and $\widetilde{\Lambda}_{g_2}=\Span_\bbZ(e_3-f_3, e_4-f_4)\simeq A_1^{\oplus2}(-1)$. In order to realize case (2) from the statement, we write $d=4d'+1$ where $d'\geq 1$ and we define $h \coloneqq 2(d'e_1+f_1)+e_3+f_3$: one can easily check that $h$ is a primitive $(2d)$-vector fixed by $g_2$. However, this time, the lattice $\bbZ h\oplus \widetilde{\Lambda}_{g_2}$ is not primitive in $\widetilde{\Lambda}$ since $\widetilde{\Lambda}$ contains 
    \[d'e_1+f_1+e_3 = \frac{h+e_3-f_3}{2}\in (\bbZ h)^\vee\oplus \widetilde{\Lambda}_{g_2}^\vee.\]
    Since $g_2$ acts trivially on $\bbZ h$ and as negative identity on $\widetilde{\Lambda}_{g_2}$, the lattice $\bbZ h\oplus \widetilde{\Lambda}_{g_2}$ has index 2 in its saturation, and the latter can thus be obtained by replacing $h$ with the vector $d'e_1+f_1+e_3$. We therefore compute:
    \[\operatorname{Sat}_{\widetilde{\Lambda}}(\bbZ h\oplus \widetilde{\Lambda}_{g_2}) = \Span_\bbZ\left(d'e_1+f_1+e_3, e_3-f_3, e_4-f_4\right)= \begin{pmatrix}
            2d'&-1&0\\-1&-2&0\\0&0&-2
        \end{pmatrix}.\]
    By applying similar arguments as in the proof of Proposition~\ref{prop:exceptions} and according to Lemma~\ref{lem: nonmodularity of exceptions}, we can thus show that case (2) occurs. For case (3), the argument is similar but this time we set $h \coloneqq 2(d'e_1+f_1)+e_3+f_3+e_4+f_4$ where $d = 4d'+2$ and $d'\geq 1$.

    For case (5), the proof works the same, by writing $d = 9d'+3$ for some $d'\geq 1$ and by setting $h\coloneqq 3(d'e_1+f_1)+e_3-f_3+2(e_4+f_4)$. \qedhere

\end{proof}

\section{Weyl reflections on modular Kummer-type manifolds}\label{sec:RefVertWall}

\subsection{Stability conditions on abelian surfaces}

Let $A$ be an abelian surface. 
We denote the derived category of 
coherent sheaves on $A$ by $\Db(A)$.  We will not give the general definition of \emph{stability conditions} on triangulated categories, directing the interested reader to \cite{MS:LecturesOnBS} for an introduction to the general theory, but instead focus on the ones we will use in this paper and the case of $\Db(A)$.

A stability condition $\sigma=(\calA_\sigma,Z_\sigma)$ consists of the heart $\calA_\sigma$ of a $t$-structure on $\Db(A)$ and the \emph{central charge} $Z_\sigma\in\Hom_\bbZ(\widetilde{H}(A,\bbZ),\bbC)$ satisfying some technical conditions, the most important of which, for our exposition, is that $Z_\sigma(v(E))\in\bbH\cup\bbR_{<0}$ for any $0\ne E\in\calA_\sigma$, where $v(E)\in\widetilde{H}(A,\bbZ)$ is the Mukai vector (simply equal to the Chern character in the case of an abelian surface).  This allows us to define the \emph{generalized slope} of $E\in\calA_\sigma$ as
$$\mu_\sigma(E):=\begin{cases}
    \frac{-\Re Z_\sigma(v(E))}{\Im Z_\sigma(v(E))}&\Im Z_\sigma(v(E))>0\\
    \infty&\Im Z_\sigma(v(E))=0
\end{cases},$$
and with it to define an object $E\in\calA_\sigma$ as \emph{$\sigma$-(semi)stable} if $\mu_\sigma(F)<\mu_\sigma(E/F)$ (resp. $\mu_\sigma(F)\le\mu_\sigma(E/F)$) for all proper subobjects $F\subset E$ in the abelian category $\calA_\sigma$.  A general object $E\in\Db(A)$ is said to be $\sigma$-(semi)stable if $E[k]\in\calA_\sigma$ for some $k\in\bbZ$ and $E[k]$ is $\sigma$-(semi)stable.
Since the Mukai pairing $(-,-)$ on $\widetilde{H}(A,\bbZ)$ is nondegenerate, there exists $w_\sigma\in\widetilde{H}(A,\bbC)$ such that $Z_\sigma(-)=(w_\sigma,-)$.

The set $\Stab(A)$ of all stability conditions admits a generalized metric \cite[Proposition 8.1]{Bri07}, and via the map
$$\eta\colon\Stab(A)\to\widetilde{H}(A,\bbC),\quad \sigma\mapsto w_\sigma,$$
whose restriction to any connected component $\Sigma\subset\Stab(A)$  is a local homeomorphism onto an open subset of a linear subspace $V(\Sigma)\subset\widetilde{H}(A,\bbC)$ \cite[Corollary 1.3]{Bri07}, the set $\Stab(A)$ is seen to be a finite dimensional complex manifold.
Furthermore, there are commuting actions of $\Aut\Db(A)$ and $\wGL2$ on $\Stabd(A)$ from the left and the right, respectively \cite[Lemma 8.2]{Bri07}.  

The stability conditions we consider in this paper lie in a distinguished connected component $\Stabd(A)\subset\Stab(A)$ characterized by the fact that for $\sigma\in\Stabd(A)$, skyscraper sheaves $\calO_x$ of points $x\in A$ are $\sigma$-stable.  More explicitly, we fix a pair $(\omega,\beta)\in \Amp(A)\times\NS(A)_\bbR$ and define 
$$Z_{\omega,\beta}(E)=(\exp(\sqrt{-1}\omega+\beta),v(E))\in\bbC$$
for every $E\in \Db(A)$.
Thus if $v(E)=(r,\theta,\chi)$, then $\Im Z_{\omega,\beta}(E)=\omega\cdot(\theta-r\beta)$.
For the heart, we define a pair of additive subcategories making up a \emph{torsion pair} (see \cite[Definition 6.1]{MS:LecturesOnBS} for the definition or \cite{HRS:Tilting} for more details) as follows, where $\langle-\rangle$ denotes the extension closure: 
\begin{align*}
    \calT^{\omega,\beta}&:=\langle E\in\Coh(A)\; |\; \text{$E$ is $\mu_\omega$-semistable with $\Im Z_{\omega,\beta}(E)>0$}\rangle\\
    \calF^{\omega,\beta}&:=\langle E\in\Coh(A)\; |\; \text{$E$ is $\mu_\omega$-semistable with $\Im Z_{\omega,\beta}(E)\le0$}\rangle
\end{align*}
From the torsion pair $(\calT^{\omega,\beta},\calF^{\omega,\beta})$, we form the following abelian subcategory, which is the heart of a $t$-structure on $\Db(A)$, by tilting $\Coh(A)$ at the torsion pair:
\begin{equation}
    \calA_{\omega,\beta}:=\{F\in\Db(A)\ |\ \calH^{-1}(F)\in\calF^{\omega,\beta},\calH^0(F)\in\calT^{\omega,\beta},\calH^i(F)=0\text{ for }i\ne-1,0\}.
\end{equation}

The following result is due to Bridgeland \cite[Section 15]{Bri08}:
\begin{theorem}
   Let $A$ be an abelian surface and fix $(\omega,\beta)\in\Amp(A)\times\NS(A)_\bbR$.  Then  $\sigma_{\omega,\beta}=(\calA_{\omega,\beta},Z_{\omega,\beta})\in\Stabd(A)$ holds, depending continuously on $(\omega,\beta)\in\Amp(A)\times\NS(A)_\bbR$.  Moreover, under the natural right action of $\widetilde{\GL}^+(2,\bbR)$, there is an identification:
$$\!^{\Stabd(A)}\!/\!_{\widetilde{\GL}^+(2,\bbR)}\simeq\Amp(A)\times\NS(A)_\bbR.$$
\end{theorem}
In other words, the explicit construction above gives a stability condition $\sigma_{\omega,\beta}$ on $\Db(A)$, and all stability conditions on $\Db(A)$ lying in the component $\Stabd(A)$ are of this form up to applying certain symmetries of $\bbC\cong\bbR^2$.

\subsubsection{Walls and moduli spaces}
Fix a Mukai vector $v\in\widetilde{H}(A,\bbZ)$.  We first recall that such a choice determines a wall-and-chamber decomposition of $\Stabd(A)$:
\begin{proposition}[{\cite[Definition 2.7]{YY}, \cite[Proposition 5.7]{MYY14}}]
There exists a locally finite set of walls (real codimension one submanifolds with boundary) in $\Stabd(A)$, depending on $v$, such that the sets of $\sigma$-stable and $\sigma$-semistable of objects of class $v$ remain constant  when $\sigma$ varies in a chamber (i.e. in a component of the complement of the walls) but necessarily change when $\sigma$ crosses a wall.  
\end{proposition}
The most important facet of the definition of the walls in $\Stabd(A)$ from the proposition is that they are the loci of $\sigma=(\calA,Z)$ for which $v$ and another, linearly independent, Mukai vector have the same $\mu_\sigma$-slope.
A stability condition is called \emph{generic with respect to $v$} if it does not lie on a wall with respect to $v$.  In particular, from the definition of these walls, when $v$ is primitive, $\sigma$ lies on a wall if and only if there exists a strictly $\sigma$-semistable object of Mukai vector $v$.

For $\sigma\in\Stabd(A)$, we denote by $M_A(v,\sigma)$ the coarse moduli space of ($S$-equivalence classes of) $\sigma$-semistable objects $E\in\Db(A)$ of Mukai vector $v$, which contains as an open subscheme $M^{st}_A(v,\sigma)$, the coarse moduli space of $\sigma$-stable objects $E\in\Db(A)$ with $v(E)=v$.

\begin{citedthm}[{(\cite[Proposition 5.16]{MYY14})}]\label{thm:Bridgelandmoduli}
    Let $A$ be an abelian surface and $v=(r,\theta,\chi)\in \widetilde{H}(A,\bbZ)$ a primitive Mukai vector with $v^2\geq 2$, and let $\sigma\in\Stabd(A)$ be general with respect to $v$.  Then the moduli space $M_A(v,\sigma)$ is a smooth projective symplectic manifold which is deformation equivalent to $\Hilb_A^{v^2/2}\times \widehat{A}$, where $\widehat{A}:=\Pic^0(A)$ is the dual abelian surface, and the Albanese morphism
$$\fraka\colon M_A(v,\sigma)\to A\times\widehat{A}$$
is an \'{e}tale locally trivial fibration. If $v^2\ge 6$ and  $X=K_A(v,\sigma)$ is a fiber of $\fraka$, then
\begin{enumerate}
    \item $X$ is a hyperkähler manifold of $\Kum_n$-type, where $2n = v^2-2$.
    \item The Mukai homomorphism \cite[Definition 1.12(2)]{Yos16}
    $$\theta_{v}\colon v^\perp\to H^2(X,\Z)$$
    is an isometry of polarized Hodge structures, where $v^\perp\subset\widetilde{H}(A,\bbZ)$ is taken with the induced polarized Hodge structure.
\end{enumerate}
\end{citedthm}
\begin{remark}
 When $v^2=0$, the moduli space $M_A(v,\sigma)$ is again an abelian surface, derived equivalent to $A$.
\end{remark}
\begin{remark}
Fix an ample divisor $H\in\Amp(A)$.  Then for a positive Mukai vector $v\in\widetilde{H}(A,\bbZ)$ (as in Definition~\ref{def:posvector}) with $v^2\ge 0$, there exists a distinguished chamber, called the \emph{Gieseker chamber}, characterized by the property that for any $\sigma$ in this chamber, $M_A(v,\sigma)\cong M_A(v,H)$ the moduli space of $H$-Gieseker semistable sheaves.
\end{remark}
    
\subsubsection{Wall-crossing and birational geometry}
A natural question is how the coarse moduli space $M_A(v,\sigma)$ changes when $\sigma$ crosses a wall in $\Stabd(A)$.  Since this moduli space is constant as $\sigma$ varies within a chamber $\calC$, we also denote it by $M_A(v,\calC)$.  Fixing a base point $\sigma$ in any chamber $\calC\subset\Stabd(A)$ with respect to $v$, the composition $\overline{\ell_\calC}=\mathrm{res}\circ\ell_\calC$ of the Bayer--Macr\`{i} map
$\ell_\calC\colon\overline{\calC}\to\NS(M_A(v,\calC))_\bbR$ from \cite{BM14a} with restriction  
$\NS(M_A(v,\calC))_\bbR\mor[\mathrm{res}]\NS(K_A(v,\calC))_\bbR$ satisfies
$$\overline{\ell_\calC}(\tau)=\theta_v\left(-\Im\left(\frac{w_\tau}{Z_\tau(v)}\right)\right),\quad \overline{\ell_\calC}(\overline{\calC})=\Nef(K_A(v,\calC)),\quad \overline{\ell_\calC}(\calC)=\Amp(K_A(v,\calC))$$
by \cite[Proposition 4.4, Remark 5.5, Lemma 9.2]{BM14a} and \cite[Proposition 3.15]{Yos16}.
The Bayer--Macr\`{i} map provides the connection to birational geometry via the following result:
\begin{proposition}[{\cite[Theorem 1.4]{BM14a}, \cite[Corollary 3.16, Proposition 3.17]{Yos16}}]
    Let $\calW\subset\Stabd(A)$ be a wall separating two chambers $\calC_\pm$, let $\sigma_0\in\calW$ be a generic stability condition (i.e. contained in no other walls), and let $\sigma_\pm\in\calC_\pm$ be two nearby stability conditions (see \cite[Remark 5.6]{BM14} for the precise meaning of \emph{nearby}).  Then $\ell_{\calC_\pm}(\sigma_0)$ (resp.   $\overline{\ell_{\calC_\pm}}(\sigma_0)$) induces a birational contraction
    $\pi^\pm\colon M_A(v,\calC_\pm)\to\overline{M}_\pm$ (resp. $\overline{\pi}^\pm\colon K_A(v,\calC_\pm)\to\overline{K}_\pm$), where $\overline{M}_\pm$ (resp. $\overline{K}_\pm$) is a normal irreducible projective variety.  The curves contracted by $\pi^\pm$ (resp. $\overline{\pi}^\pm$) are precisely the curves of $\sigma_\pm$-stable objects that are $S$-equivalent to one another with respect to $\sigma_0$.
\end{proposition}
We classify walls in terms of the corresponding birational contraction $\overline{\pi}^\pm$ as follows:
\begin{definition}
    We say that a wall $\calW$ is:
    \begin{enumerate}
        \item a \emph{fake wall} if $\overline{\pi}^\pm$ contracts no curves;
        \item a \emph{flopping wall} if we can identify $\overline{K}_-=\overline{K}_+$ and $\overline{\pi}^\pm$ is a flopping contraction;
        \item a \emph{divisorial wall} if $\overline{\pi}^\pm$ are both divisorial contractions. 
    \end{enumerate}
    While the above types classify the wall $\calW$ in terms of the behavior of the corresponding birational contractions $\overline{\pi}^\pm$, it is useful to say that $\calW$ is 
    \begin{enumerate}[resume]
        \item a \emph{totally semistable wall} if $M_A^{st}(v,\sigma_0)=\varnothing$ for generic $\sigma_0\in\calW$.
    \end{enumerate}
\end{definition}
The type of a wall $\calW$ can be determined purely lattice-theoretically.  Given a wall $\calW$, one can associate to it the hyperbolic rank $2$ primitive sublattice $\calH_\calW\subset\widetilde{H}^{1,1}(A,\bbZ)$ given by 
\begin{equation}
    \calH_\calW=\left\{w\in\widetilde{H}^{1,1}(A,\bbZ) \;|\;\Im\frac{Z_\sigma(w)}{Z_\sigma(v)}=0\text{ for all }\sigma\in\calW\right\}.
\end{equation}
Conversely, given a primitive rank $2$ hyperbolic sublattice $\calH\subset\widetilde{H}^{1,1}(A,\bbZ)$ containing $v$, one can define a \emph{potential wall} $\calW$ associated to $\calH$ as a connected component of the real codimension one submanifold of $\sigma\in\Stabd(A)$ such that $Z_\sigma(\calH)$ is contained in a line.  We summarize the lattice-theoretic classification of walls in the following result:
\begin{theorem}\label{thm:ClassificationWalls}
Let $v\in\widetilde{H}^{1,1}(A,\bbZ)$ be a primitive Mukai vector  with $v^2\ge 6$. Let $\sH\subset\widetilde{H}^{1,1}(A,\bbZ)$ be a primitive hyperbolic rank $2$ sublattice containing $v$ and let $\sW\subset\Stabd(A)$ be a potential wall associated to it.
    \begin{enumerate}[leftmargin=*, noitemsep]
        \item The set $\sW$ is a divisorial wall if and only if there exists a primitive isotropic class $w\in\sH$ such that $(w,v)=1$ or $2$.
        \item The set $\sW$ is a flopping wall if and only if the conditions in $(1)$ are not satisfied and there exists a class $v_1\in\sH$ satisfying 
        \begin{equation}\label{eqn:wall-def}
             (v_1,v-v_1)>0,v_1^2\ge 0,(v-v_1)^2\ge 0.
        \end{equation}
        \item In all other cases, $\sW$ is not a wall.
    \end{enumerate}
\end{theorem}
\begin{proof}
For $\sigma_0\in\calW$ not on any other walls for $v$, let $\sigma_\pm$ be two nearby stability conditions in opposite chambers separated by $\calW$. 
For ease of notation, let us denote $K^+\coloneqq K_A(v,\sigma_+)$, $K^-\coloneqq K_A(v,\sigma_-)$ and $K^0\coloneq K_A(v,\sigma_0)$.
We denote by $K^*:=K^\pm\cap M_A^{st}(v,\sigma_0)$.  By the openness of stability, $M_A^{st}(v,\sigma_0)$ is an open subset of $M_A^{st}(v,\sigma_\pm)$, so $K^*$ is an open subset of $K_A^\pm$.  In particular, $\calW$ is a totally semistable wall if and only if this open subset of $K^\pm$ is empty. Moreover, the exceptional locus of $\overline{\pi}^\pm$ is contained in $K^\pm\backslash K^*$ by \cite[Proposition 3.17]{Yos16}.

Thus, if a wall $\calW$ is divisorial then $K^\pm\backslash K^*$ contains a divisor, i.e.~$\codim\left(K^\pm\backslash K^*,K^\pm\right)\le 1$.  This occurs if and only if there exists a primitive isotropic class $w\in\sH$ such that $(w,v)=1$ or $2$ by \cite[Propositions 3.27 and 3.29]{Yos16}.  Moreover, it is shown there that in each of these cases $\overline{\pi}^\pm$ is a divisorial contraction.

Using (1), we conclude that if the conditions in $(1)$ are not satisfied,
     then we must have $\codim\left(K^\pm\backslash K^*,K^\pm\right)\ge 2$.  
Now assume that there exists a class $v_1\in\sH$ satisfying 
     $$(v_1,v-v_1)>0,v_1^2\ge 0,(v-v_1)^2\ge 0.$$
Then we claim that $K^*\setminus K^\pm\neq\emptyset$, which implies that $\sW$ is a flopping wall.
To this end, note that, by \cite[Lemma 9.2]{BM14} we may assume that the parallelogram in $\sH\otimes\bbR$ has no other lattice points in $\sH_\sW$ other than 
     its vertices $0$, $v_1$, $v-v_1$, and $v$.  In particular, we may assume that                                         $v_1$ and $v-v_1$ are both primitive.  Up to switching $\sigma_+$ and $\sigma_-$, we may assume that $\mu_{\sigma_+}(v_1)<\mu_{\sigma_+}(v-v_1)$.  Since $\sigma_\pm$ is by definition nearby, it is also generic with respect to $v_1$ and $v-v_1$, by Theorem~\ref{thm:Bridgelandmoduli}, there exists $\sigma_+$-stable objects $A_1$ and $A_2$ of Mukai vectors $v_1$ and $v-v_1$, respectively.
     In addition, we have $(v_1,v-v_1)>2$.  Indeed, if $v_1^2=0$, then $(v_1,v-v_1)=(v,v_1)>2$ by the assumption that the conditions in (1) are not satisfied.  And if $v_1^2>0$, and by symmetry, we may assume $(v-v_1)^2>0$ as well, then by the signature of $\sH_\sW$, we have 
$$(v_1,v-v_1)>\sqrt{v_1^2(v-v_1)^2}\ge 2$$
since $\sH_\sW$ is an even lattice.  It follows from this that 
$$\ext^1(A_2,A_1)=(v_1,v-v_1)+\hom(A_2,A_1)+\hom(A_1,A_2)\ge(v_1,v-v_1)>2.$$
Any non-trivial extension in $\Ext^1(A_2, A_1)$
%$$0\to A_1\to E\to A_2\to 0$$
is $\sigma_+$-stable of Mukai vector $v$ by \cite[Lemma 9.3]{BM14}, and as all these extensions are $S$-equivalent to each other with respect to $\sigma_0$, there is a projective space $\bbP=\bbP\Ext^1(A_2,A_1)$ of dimension at least two that gets contracted by $\pi^+$.  As the restriction $\fraka|_{\bbP}$ of the \'{e}tale locally trivial Albanese morphism $\fraka$ to $\bbP$ is  constant, we must have $\bbP\subset K^+$ is thus contracted to a point by $\overline{\pi}^+=\pi^+|_{K^+}$.   As the exceptional locus of this birational morphism is contained in $K^+\backslash K^*$, we obtain the desired nonemptyness.

If the numerical conditions in $(1)$ and $(2)$ are both not satisfied, then $\sW$ is not a wall at all as every $\sigma_+$-stable object $E$ of Mukai vector $v$ is still $\sigma_0$-stable, and thus $\sigma_-$-stable by openness of stability.  Indeed, otherwise some $\sigma_+$-stable object $E$ of Mukai vector $v$ is strictly $\sigma_0$-semistable and thus $\sigma_-$-unstable.  Taking the penultimate Harder-Narasmihan filtration subobject $S\subset E$ and the final HN-factor $Q$, we see that $E$ is an extension of $Q$ by $S$.
As $v(S)$ is the sum of Mukai vectors in the same component of the positive cone of $\calH$ as $v$, it follows that $v(Q)^2\ge0$, $v(S)^2\ge 0$ and $(v(Q),v(S))>0$ by the same signature argument above, contradicting the assumption that the conditions in $(1)$ and $(2)$ are not satisfied.
\end{proof}
By varying over all chambers with respect to $v$ and studying the wall-crossing contractions $\pi^\pm$ and $\overline{\pi^\pm}$, Yoshioka proved the following result showing that wall-crossing in $\Stabd(A)$ runs the full minimal model program on $K_A(v,\sigma)$:
\begin{theorem}[{\cite[Theorem 3.31, Corollary 3.33]{Yos16}}]\label{thm:BridgelandmoduliBirational}
    Let $A$ be an abelian surface, $v\in\widetilde{H}(A,\bbZ)$ a primitive Mukai vector with $v^2\ge 6$, and $\sigma\in\Stabd(A)$ generic with respect to $v$.  
    \begin{enumerate}
    \item For any other stability condition $\tau\in\Stabd(A)$ generic with respect to $v$, $M_A(v,\tau)$ (resp. $K_A(v,\tau)$) is birational to $M_A(v,\sigma)$ (resp. $K_A(v,\sigma)$).  Upon identifying $\NS(M_A(v,\calC_+))$ and $\NS(M_A(v,\calC_-))$ (resp. $\NS(K_A(v,\calC_+))$ and $\NS(K_A(v,\calC_-))$) when crossing a wall $\calW$ separating $\calC_+$ and $\calC_-$, the maps $\ell_{\calC_\pm}$ (resp. $\overline{\ell_{\calC_\pm}}$)  then glue to continuous maps
\begin{equation}
\overline{\ell}\colon\Stabd(A)\mor[\ell]\NS(M_A(v,\sigma))_\bbR\mor[\mathrm{res}]\NS(K_A(v,\sigma))_\bbR
\end{equation}
such that the image of $\overline{\ell}$ is $\Mov(K_A(v,\sigma))$.
\end{enumerate}
More precisely, we have the following:
\begin{enumerate}[resume]
        \item Let $(K,L)$ be a pair of a smooth manifold $K$ with $\omega_K\cong\calO_K$ and an ample divisor $L$ on $K$.  If $K$ is birationally equivalent to $K_A(v,\sigma)$, then there is some $\tau\in\Stabd(A)$, generic with respect to $v$, such that $K\cong K_A(v,\tau)$ such that $\ell(\tau)\sim_{\bbR_{>0}} L$, where $\sim_{\bbR_{>0}}$ means proportional by a factor in $\bbR_{>0}$.
        \item Let $(M,L)$ be a pair of smooth manifold $M$ with $\omega_M\cong\calO_M$ and an ample divisor $L$ on $M$.  If $M$ is birational to $M_A(v,\sigma)$, then there is $\tau\in\Stabd(A)$, generic with respect to $v$, such that $M\cong M_A(v,\tau)$ and $\ell(\tau)\sim_{\bbR_{>0}}L$ up to divisors coming from the Albanese variety $\Alb(M_A(v,\tau))$.
    \end{enumerate}
\end{theorem}
\subsection{Reflection in the vertical wall}
\subsubsection{Reducing to moduli of sheaves}
With precisely classified exceptions, we know by Theorem~\ref{thm:main theorem} that if a projective $\Kum_n$-type manifold $K$ for $n\ge 2$ admits a finite order symplectic birational self-map, then there is a birational map  $\varphi\colon K\dashrightarrow K_{A,\alpha}(v,H)$ for some abelian surface $A$, Brauer class $\alpha\in\Br(A)$, polarization $H$ and primitive Mukai vector $v\in\widetilde{H}(A,\bbZ)$ with $v^2\ge 6$. We focus on the case where $\alpha=0$, in particular we get $K\cong K_A(v,\sigma)$ for some $\sigma\in\Stabd(A)$ generic with respect to $v$.
In the next result, we show that we may assume $K$ is the Albanese fiber of a moduli space of stable sheaves with positive rank on $A$ or $\widehat{A}$, and then we construct such a finite order symplectic birational self-map on $K$.

\begin{lemma}\label{lem:positive rank}
Let $A$ be an abelian surface and $v\in\widetilde{H}(A,\hz)$ be a primitive Mukai vector with $v^2\ge 6$.  
Then for any generic stability condition $\sigma\in\Stabd(A)$ with respect to $v$ there exists a birational map $K_A(v,\sigma)\dashrightarrow K_T(v',H)$, where $T\in\{A,\widehat{A}\}$, the vector $v'\in\widetilde{H}(T,\bbZ)$ has positive rank, and $H$ is an ample divisor which is generic with respect to $v$.
\end{lemma}
\begin{proof}
We show that by applying an appropriate sequence of derived equivalences, we get a Mukai vector $v'$ of positive rank on either $A$ or $\widehat{A}$.  Indeed, just as autoequivalences act on $\Stab(A)$, if $\Phi\colon\Db(A)\to\Db(T)$ is a derived equivalence, where $T$ is another abelian surface, then we get a corresponding isometry of Mukai lattices (\cite[Corollary 9.43]{HuyFM}) which we write $\Phi_*\colon\widetilde{H}(A,\bbZ)\to\widetilde{H}(T,\bbZ)$.
Together these induce isomorphisms
$$\Phi\colon \Stab(A)\isomor\Stab(T),\quad M_A(v,\sigma)\isomor M_{T}(\Phi_*(v),\Phi(\sigma)),\quad K_A(v,\sigma)\isomor K_{T}(\Phi_*(v),\Phi(\sigma)).$$
In particular, if $\Phi$ sends $\Stabd(A)$ to $\Stabd(T)$ and $v'=\Phi_\ast(v)$ has positive rank, then choosing a generic ample divisor $H$ and $\tau$ in the Gieseker chamber for $v'$ with respect to $H$, then by Theorem~\ref{thm:BridgelandmoduliBirational} we have a birational map 
$$K_A(v,\sigma)\isomor K_T(\Phi_*(v),\Phi(\sigma))\dashrightarrow K_T(\Phi_*(v),\tau)\cong K_T(v',H).$$
As the action of $\Aut(\Db(A))$ on $\Stab(A)$ preserves the component $\Stabd(A)$ (see \cite[Section 15]{Bri08}), we may freely take $T=A$ and $\Phi\in\Aut(\Db(A))$.

Now, write $v=(r,\theta,\chi)$.   
If $r\ne 0$, then taking $T=A$ we may assume $r>0$ up to applying $\Phi=[1]$.
So we may assume that $r=0$.  

If $r=0$ but $\chi\ne 0$, then we apply the Fourier--Mukai transform introduced by Mukai \cite{Mukai:FMDuality}.  Denote by $\Pp$ the Poincar\'{e} line bundle on $A\times\widehat{A}$, and let $\Phi_{A \to \widehat{A}}^{{\Pp}}:\Db(A) \to \Db(\widehat{A})$
be the Fourier--Mukai transform whose kernel is ${\Pp}$:
\begin{equation}
\begin{matrix}
\Phi_{A \to \widehat{A}}^{{\Pp}}: & \Db(A) & \to & \Db(\widehat{A})\\
& E & \mapsto & {\bf R}p_{\widehat{A}*}({\Pp} \otimes p_A^*(E))
\end{matrix},
\end{equation}
where $p_A:A \times \widehat{A} \to A$ and
$p_{\widehat{A}}:A \times \widehat{A} \to \widehat{A}$
are the two projections.
Then 
$\Phi:=\Phi_{A \to \widehat{A}}^{{\Pp}}$ is a derived equivalence which induces an isometry of Hodge structures
\begin{equation}\label{eqn:FMabeliansurfanddual}
\begin{matrix}
\widetilde{H}(A,\hz) & \to & \widetilde{H}(\widehat{A},\hz)\\
(m,\xi,a) & \mapsto & (a,-\widehat{\xi},m)
\end{matrix},
\end{equation}
where $\widehat{\xi}$ is the Poincar\'{e} dual of $\xi$ by \cite[Proposition 9.19, Remark 9.20,Lemma 9.23]{HuyFM}. 
As $\Phi$ sends $\Stabd(A)$ to $\Stabd(\widehat{A})$ by \cite[Theorem 3.8(1)]{YY}, we may apply $\Phi$ to reduce to the previous case, in which case $T=\widehat{A}$ and $v'=\pm \Phi_*(v)$ depending on the sign of $\chi$.
So we may assume that $r=\chi=0$.

Since $r=\chi=0$, but $v^2>0$, we must have $\theta\ne 0$.  Tensoring by a line $L\in\Pic(A)$ such that $\theta\cdot L\ne 0$ gives an autoequivalence of $\Db(A)$ that sends $v$ to $(0,\theta,\theta\cdot L)$, so may reduce to the previous case.
\end{proof}
Lemma~\ref{lem:positive rank} lets us reduce to the case constructing finite order symplectic birational self-maps on the Albanese fiber of a moduli space of sheaves on an abelian surface $A$.
We construct one in this case in the next subsection.

\subsection{The vertical wall}
As in the proof of the lemma, write the vector $v$ as $v=(r,\theta,\chi)$, with $r,\chi\in \bbZ$, and $\theta=cD\in \NS(A)$, with $D$ a primitive class. According to \cref{lem:positive rank}, we assume that $r>0$.  Let $K\coloneqq K_A(v,H)$ be the Albanese fiber of a moduli space of $H$-stable sheaves for a $v$-generic ample divisor $H$.

\begin{definition}\label{def:vertical wall}
    We define the \emph{vertical wall} of $A$ associated to $v$ to be the set
    \[\calW\coloneqq\{(\omega,\beta)\in \Amp(A)\otimes \NS(A)_{\mathbb{R}}\,\mid\, Z_{\omega, \beta}(v) \in \bbR\}.\]
\end{definition}
In other words, $\calW$ is the set of $(\omega,\beta)$ satisfying $\omega\cdot\theta=r\omega\cdot\beta$.  We consider the corresponding hyperbolic rank two lattice
$$\calH_\calW=\{w=(\lambda,(\lambda/r)\theta,\mu)\ |\ \lambda,\mu\in\Z,\quad r\mid c\lambda\},$$
the saturation in $\widetilde{H}(A,\bbZ)$ of the sublattice spanned by $v$ and $(0,0,1)$.  Equivalently, this is the sublattice of $w\in \widetilde{H}(A,\bbZ)$ such that $Z_{\omega,\beta}(w)\in\bbR$ for all $(\omega,\beta)\in\calW$. The computations of \cite[Lemma 9.2]{BM14a} are identical for abelian surfaces, so $\overline{\ell}(\sigma_{\omega,\beta})\sim_{\bbR_{>0}}\theta_v(a_{\omega,\beta})$ for $(\omega,\beta)\in \calW $, where $a_{\omega,\beta}=(0,\omega,\beta\cdot\omega)=(0,\omega,\omega\cdot\theta/r)$.  The previous vectors form a cone $C_\calW$ of codimension one, that is, a wall, in the positive cone $\calC_K\subset\NS(K)_\bbR\simeq v^\perp_\bbR\subset\widetilde{H}^{1,1}(A,\bbZ)_\bbR$.  One of the chambers adjacent to $C_\calW$ is the Gieseker chamber.  

We would like to consider the reflection in this wall.  Via $\theta_v^{-1}$, this wall consists of all vectors in $\widetilde{H}^{1,1}(A,\bbZ)_\bbR$ that are orthogonal to both $v$ and $a_{\omega,\beta}$ for all $(\omega,\beta)\in\calW$.  The latter condition means that any such vector $(R,C,S)$ must satisfy 
$$0=((R,C,S),(0,\omega,\omega\cdot\theta/r))=C\cdot\omega-\frac{R}{r}\theta\cdot\omega$$
as $\omega$ varies over the ample cone so that any $\omega\in\Amp(A)$ satisfies $(R\theta-rC)\cdot\omega=0$. As $\Amp(A)$ is a full dimensional open cone in $\NS(A)_\bbR$, it follows that $R\theta-rC=0$ holds, so using orthogonality to $v$, we see that $(R,C,S)$ must be proportional to $e\coloneqq(r,\theta,\frac{\theta^2}{r}-\chi)$.  
The reflection in the wall $C_\calW$ is thus $R_e$, and using \cref{Prop:Reflection in Mon}, we can now determine precisely when this reflection is in $\MonH(K)$.
\begin{lemma}
    The reflection $R_e$ belongs to $\MonH(K)$ if and only if 
    \begin{equation}\label{eqn:numerical criteria on reflection in Mon2}
   r\mid 2c  \ \text{ and } \ \gcd(r,\chi)=1 \text{ or } 2.
    \end{equation}
\end{lemma}
\begin{proof}The proof works exactly the same way as in \cite[Lemma 4.13]{DMPM}.
\end{proof}

     \begin{proposition}\label{prop:walldivisorialorflopping}
         Let $\calW$ be the vertical wall of $A$ associated to a primitive Mukai vector $v=(r,cD,\chi)$ with $v^2\ge 6$. If $v$ satisfies the conditions in \eqref{eqn:numerical criteria on reflection in Mon2}, then the corresponding wall in $\Stabd(A)$ is divisorial or flopping.
     \end{proposition}

\begin{proof}
Suppose first that $v^2\ge 2r$.  Then consider the decomposition $v=w+t$, where $w=v+(0,0,1)$ and $t=(0,0,-1)$, so that $(w,t)=r$, $w^2=v^2-2r$ and $t^2=0$. Thus $\calW$ is either divisorial or flopping by (\ref{eqn:wall-def}). 

Now, assume that $0<v^2<2r$. Since $r\mid 2c$, we get $r\mid v^2=c^2D^2-2r\chi$, in particular $v^2=r$. Write $2c=mr$. Then
\begin{equation}\label{eqn:v^2/r}
    \dfrac{v^2}{r}=1=mc\dfrac{D^2}{2}-2\chi.
\end{equation}
We obtain that $mc\dfrac{D^2}{2}$ is odd, hence $r=2a$ is even (follows from $2c=mr$) and $a$ is odd (follows from $c=ma$ odd). Isolating $\chi$ in (\ref{eqn:v^2/r}), it gives
$$v=\left(2a,maD ,\frac{D^2m^2a-2}{4}\right)$$
which is precisely the last case listed in Proposition \ref{prop:when the vertical wall divisorial for v}. Explicitly, the integral vector $w=\left(4,2mD,\dfrac{D^2m^2}{2}\right)\in \widetilde{H}^{1,1}(S,\bbZ)$ satisfies $(v,w)=2$.
\end{proof}
\begin{remark}\label{rem:geometric_wall_crossing}
While the lattice-theoretic proof above is efficient, it is conceptually illuminating to consider the underlying geometric wall-crossing behavior. By the standard classification of walls, demonstrating that $\calW$ is an actual (non-fake) wall requires constructing a curve in $K$ that parameterizes non-isomorphic $H$-stable sheaves of class $v$ which become $S$-equivalent with respect to the stability condition $\sigma_0 = \sigma_{H, \beta}$ for $(H, \beta) \in \calW$. Since the vertical wall induces the Hilbert--Chow morphism when $r=1$, a well-known divisorial contraction, we may assume that $r \ge 2$.

As demonstrated by Tajakka \cite[Theorem 7.1]{Tajakka}, crossing the vertical wall always induces the Li--Gieseker--Uhlenbeck contraction morphism to the Uhlenbeck compactification of the moduli space of $\mu_H$-stable locally free sheaves. Furthermore, the $\sigma_0$-stable objects with infinite slope are precisely the $\mu_H$-stable locally free sheaves, alongside shifted skyscraper sheaves $\mathcal{O}_x[-1]$ for points $x \in A$ \cite[Proposition~2.2]{Tajakka}. It follows that the $\mu_H$-stable sheaves of class $v$ that become strictly $\sigma_0$-semistable must be either non-locally-free or strictly $\mu_H$-semistable. The dichotomy between the two cases considered in the proof of Proposition~\ref{prop:walldivisorialorflopping} dictates which of these two types of objects populates the contracted exceptional locus:

\begin{itemize}[leftmargin=*]
    \item \textbf{Case 1 ($v^2 \ge 2r$):} In this ``large discriminant'' regime, the exceptional locus is dominated by non-locally-free sheaves (arising conceptually as kernels of surjections from a bundle to a skyscraper sheaf). 
    
    \item \textbf{Case 2 ($v^2 < 2r$):} In this ``small discriminant'' regime, under the numerical conditions $v^2 \ge 6$, $r | 2c$, and $\gcd(r, \chi) \in \{1, 2\}$, the rank $r$ must be even and there is a decomposition $v=v_1+v_2$ with
\[v_1=\left(2,mD,\frac{m^2d-1}{2}\right) \text{ and } v_2=\left(r-2,\left(\frac{r}{2}-1\right)mD,\frac{(r-2)m^2d}{4}\right).\]
One checks easily that
$v_1^2=2$, $v_2^2=0$, and $\langle v_1,v_2\rangle=\frac{r}{2}-1\ge 2$. Thus we may take $\mu_H$-stable locally free sheaves $E_i$ of Mukai vector $v_i$ and all non-trivial extensions $E\in\Ext^1(E_2,E_1)$ are Gieseker stable but $S$-equivalent at the vertical wall.  Thus the entire positive dimensional projective space $\bbP\Ext^1(E_2,E_1)$ gets contracted.
Unlike the first case, the contraction here does not address point-like singularities of the sheaves. 
\end{itemize}
\end{remark}

We call $\sW$ divisorial or flopping in accordance with the type of the corresponding wall in $\Stabd(A)$.

\begin{proposition}\label{prop:when the vertical wall divisorial for v}
Let $\calW$ be a vertical wall of $A$ associated to a primitive Mukai vector $v=(r,cD,\chi)$ satisfying condition \eqref{eqn:numerical criteria on reflection in Mon2}. Then $\calW$ is a divisorial wall if and only if one of the following conditions hold:
    \begin{enumerate}
        \item $r\le 2$
        \item $v=(r,krD,\frac{D^2k^2r}{2}-m)$ for some $k\in\Z$ and $m=1$ or $2$.
        \item $r>2$, $r\nmid c$, and one of the two possibilities occurs.
        \begin{enumerate}
            \item $D^2\equiv 0\pmod 4$, $v=(2a,maD,\frac{D^2m^2a}{4}-1)$ for some integers $a\ge 2$, $m$ odd.
            \item $D^2\equiv 2\pmod 4$, $v=(2a,m'aD,\frac{D^2m'^2a-2}{4})$ for some odd integers $a\ge 3$, $m'$.
        \end{enumerate}
    \end{enumerate}
\end{proposition}
\begin{proof}
    The proof is identical to that of \cite[Proposition 4.15]{DMPM} except that the spherical cases need not be considered.   
\end{proof}
It follows from Proposition~\ref{prop:walldivisorialorflopping} that if $v$ does not satisfy any of the conditions of Proposition~\ref{prop:when the vertical wall divisorial for v}, then $R_e\in\Mon_{\rm Bir}^2(K)$.  We are now in position to prove \cref{theorem:main theorem 2-explicit examples}. 
\begin{proof}[Proof of \cref{theorem:main theorem 2-explicit examples}]
Assume that $v$ satisfies  \eqref{eqn:numerical criteria on reflection in Mon2} so that $R_e\in \Mon_{\rm Hdg}^2(K)$.  Then by Theorem~\ref{theo:semidirect decomposition}, it decomposes uniquely as 
 $R_e= \psi\circ g^*$,
with $\psi\in W_{\rm Exc}(K)$ and $g\in \Bir(K)$ a symplectic birational self-map, unique up to the subgroup of cohomologically trivial automorphisms $\ker \rho_K$ (see \eqref{eq:natural representation}).
In view of  \cref{prop:walldivisorialorflopping}, all we have to show is that $\psi=\Id$ exactly when $\calW$ is a flopping wall. If $\calW$ is divisorial, then it lies on the boundary of $\Mov(K)$. By Definition \ref{def:wallchamberCX},
%\cite[Theorem 6.17 and Lemma 6.22]{Marsurvey}, 
the equality $\calW=E^\perp$ holds for some prime exceptional divisor $E$, in particular $R_e=\psi$ lies in $W_{\Exc}(K)$ and $g=\Id$. On the other hand, when $\calW$ is a flopping wall, it lies in the interior of $\Mov(K)$. Therefore, one can find $D\in \Mov(K)$ such that $(g^{-1})^*(D)$ lies close to $C_\calW$. Then we have $R_e\circ (g^{-1})^*(D)=\psi(D)\in \Mov(K)$. This is only possible when $\psi=\Id$ since $\Mov(K)$ is a fundamental domain for the faithful action of $W_{\Exc}(K)$ (Theorem~\ref{theo:semidirect decomposition}).  Thus $R_e=g^*$. The equality $g^*|_{A_M}=-\Id$ is a consequence of \cref{Prop:Reflection in Mon}.
\end{proof}

\begin{example}\label{example:MarkmanExample}
When $r>2$, we recover an analog of Markman's examples \cite[Section 11]{MarkmanPex}. In loc. cit., Markman considers a polarized K3 surface $(S,H)$ and the Mukai vectors $v=(r,0,s)$ (resp. $v=(2r,rH,-b)$) for some values $r,s,b\in \bbZ$ satisfying certain properties. He then proves that the moduli space $M_S(v)$ admits a birational involution generically described by $\calE\mapsto \calE^\vee$ (resp. $\calE\mapsto \calE^\vee\otimes \calO(H)$), whose cohomological action is exactly the reflection in the vertical wall. 

These operations and Markman's computations hold and generalize to $K_A(v,H)$: for any $v=(r,cD,\chi)$ with $r\geq 1$ and $r\mid 2c$, i.e. $c=\frac{mr}{2}$, 
the reflection $R_e$ in the vertical wall is induced by the functor
\[\calE \mapsto \calE^\vee \otimes \calO(mD).\]
The conditions under which the reflection is divisorial of flopping are determined by Proposition~\ref{prop:when the vertical wall divisorial for v}.
\end{example}

\appendix
\section{Moduli spaces of twisted sheaves}\label{sec:modulispaces} Due to lack of references, in this appendix, we define the twisted modular Kummer-type manifolds $K_{A,\alpha}(v, H)$ as the Albanese fiber of the moduli spaces of $H$-stable $\alpha$-twisted sheaves with Mukai vector $v$ on abelian surfaces $A$. More precisely, we aim to prove Theorem \ref{thm:twistedYoshioka}. In this section, we use the notation $M_{A,B}(v,H)$ rather than $M_{A,\alpha}(v,H)$ (and similarly for the Albanese fiber) to emphasize the choice of the $B$-field, see Section \ref{sec:AbSurf}. We work over $\bbC$.

The proof uses the constructions explained in the next sections. More precisely, we deform $A$, $\alpha$, $H$ and $K_{A,B}(v,H)$ and the Mukai homomorphism $\theta_{v}$ to another abelian surface $A'$ over which $\alpha$ becomes trivial, so that $K_{A'}(v,H')$ parametrizes non-twisted sheaves. Since the result has been shown by Yoshioka in the case $\alpha=0$ \cite[Theorem 0.2]{YosModuliSheaves}, the theorem follows.

\subsection{Twisted sheaves and moduli spaces}

For an extensive introduction to twisted sheaves, we refer to \cite{CaldararuPhD,VanBreePhD}. Given a smooth projective variety $X$ over an algebraically closed field and a class $\alpha\in \Br(X)$, defining $\alpha$-twisted sheaves and their moduli spaces usually requires an extra datum, which usually amounts to lift the class $\alpha$ through the morphism $H^1(X,{\rm PGL}_k) \to H^2(X,\bbG_m)$ or $H^2(X,\mu_k) \to H^2(X,\bbG_m)$.
\begin{enumerate}
    \item One can fix an Azumaya algebra $\mathfrak{A}$ that represents the class $\alpha$. Twisted sheaves are then defined as left $\mathfrak{A}$-modules. \label{eq:Azumayatwist} 
    \item One can fix a Brauer-Severi variety $Y\to X$ representing $\alpha$, and consider twisted sheaves as certain sheaves on $Y$. \label{eq:BStwist}
    \item One can fix a $\bbG_m$-gerbe (or, as often needed, a $\mu_k$-gerbe) $\mathfrak{X}\to X$, and consider twisted sheaves as sheaves of weight $1$ on $\mathfrak{X}$. \label{eq:gerbetwist}
\end{enumerate}

The moduli space of semistable twisted sheaves was constructed in setting (\ref{eq:Azumayatwist}) by Simpson \cite[Theorem 4.7]{SimpsonModuliTwist}, see also \cite{HoffmannStuhler}; in setting (\ref{eq:BStwist}) by Yoshioka \cite{YoshiokaTwistedSheaves} and in setting (\ref{eq:gerbetwist}) by Lieblich \cite{LieblichTwistedSheaves}. These constructions are equivalent \cite[Theorem 3.2.9]{VanBreePhD}. Before considering the case of abelian surfaces, we state the general existence result. We use the form that appears (and is reproven) in \cite[Theorem 3.2.11]{VanBreePhD}. In the following statement, twisted sheaves are considered as in setting (\ref{eq:gerbetwist}).

\begin{theorem}\label{thm:modulitwistgeneral}
    Let $f\colon X\to S$ be a projective morphism, with $S$ an affine scheme of finite type over an algebraically closed field of characteristic $0$. Fix $H$ a $f$-relative polarization, $\mathfrak{X}\to X$ a $\mu_k$-gerbe and $P$ a polynomial. Then, there exists a relative moduli stack $\mathfrak{M}^P_{\mathfrak{X}/S} \to S$ parametrizing families of twisted semistable sheaves on $X$ with twisted Hilbert polynomial $P$. Moreover, there exists a projective coarse moduli space $M^P_{\mathfrak{X}/S}\to S$. Over the open stable locus $M^{P,\text{st}}_{\mathfrak{X}/S}\subset M^P_{\mathfrak{X}/S}\to S$, there exists a universal family \'etale-locally and $\mathfrak{M}^{P,\text{st}}_{\mathfrak{X}/S}\to M^{P,\text{st}}_{\mathfrak{X}/S}$ is a $\bbG_m$-gerbe.
\end{theorem}

The twisted Hilbert polynomial is defined as in the non-twisted case, see \cite[Section 3.2.1]{VanBreePhD}. In particular, one uses that $\mu_k$-gerbe are Deligne--Mumford stacks, hence admit a theory of Chow groups and Chern character (see also \cite{Heinloth, ReiAutoequiv}).

Finally, note that it is often useful in practice to fix a Chern character ${\rm ch}\in \bigoplus_{p}\Gamma(S,\mathbf{R}^pf_*\bbQ)$ and consider the moduli space $M_{\mathfrak{X}/S}({\rm ch})$ of twisted sheaves with fixed Chern character, rather than fixed Hilbert polynomial.

\subsection{Abelian surface}\label{sec:AbSurf}

Let $A$ be an abelian surface and $\mathfrak{X}\to A$ a $\mu_k$-gerbe. Choose a $B$-field lift $B\in H^2(A,\bbQ)$ of the class $[\mathfrak{X}]\in H^2(A,\mu_k)$. The choice of $B$ determines a \v{C}ech-cocycle $(\alpha_{ijk})\in \check{H}^2(A,\calO^*_A)_{\rm tors}$. Huybrechts and Stellari \cite{HSEquivTwist} defined a $B$-Mukai vector
\begin{equation}\label{eq:Bmukaiapp}
    v_B\colon \mathbf{D}^b(A,(\alpha_{ijk})) \to \widetilde{H}(A,B,\bbZ),
\end{equation}
see Section \ref{sec:DefExamples} for the definition of the $B$-twisted Hodge structure.
Given a twisted sheaf $\calF$, there is a comparison with the Chern character constructed on $\mathfrak{X}$:
$v_B(\calF)=\ch_{\mathfrak{X}}(\calF)\cdot e^{B}$, see \cite[Proposition 3.1.71]{VanBreePhD}. We use the two notions and notations interchangeably:
$$M_{\mathfrak{X}}({\rm ch}) =: M_{A,B}(v).$$

\begin{proposition}\label{prop:modulitwistsmoothproj}
    Let $A$ be an abelian surface.
    Let $v\in \widetilde{H}(A,B,\bbZ)$ be a primitive positive $B$-Mukai vector (Definition \ref{def:posvector}). For a general polarization $H$, all semistable sheaves are stable and the moduli space $M_{A,B}(v, H)$ is smooth and projective of dimension $v^2+2$.
\end{proposition}

\begin{proof}
First, we claim that for $H$ chosen $v$-generic, all semistable sheaves are stable. Note that by the comparison $v_B(-)=\ch_\mathfrak{X}(-)\cdot e^B$, we may use the $B$-Mukai vector to compute the Hilbert polynomial, which in turns can be computed via Hirzebruch--Riemann--Roch \cite[Definition 3.2.1]{VanBreePhD}. In this situation, the proof in the non-twisted case applies identically. As a consequence, $M_{A,B}(v,H)=M_{A,B}^{\rm st}(v,H)$ is projective.
Therefore, we are left to prove smoothness of the moduli space at a point corresponding to a stable twisted sheaf. This is stated in \cite{YoshiokaTwistedSheaves} in the case of K3 surfaces.

    We make a choice of a Brauer--Severi variety $p\colon Y\to A$ corresponding to the lift of $[\mathfrak{X}]$ to a class in $H^1(A, {\rm PGL}_k)$, compatible with the choice of the $B$-field $B$
    (see [Lemma 1.2, 1.3; \textit{loc.\ cit.}]). 
    From [Lemma 2.2; \textit{loc.\ cit.}] the independence of the definition of the moduli space from the choice of the Brauer--Severi variety $Y$ follows.
    We define an $\alpha$-\emph{twisted $Y$-sheaf} to be a non-twisted coherent sheaf on $Y$ given by the pullback of an $\alpha$-twisted sheaf on $X$ untwisted by a fixed $p^*\alpha^{-1}$-twisted line bundle on $Y$, where $\alpha$ is seen as a \v{C}ech-cocycle $\alpha\in \check{H}^2(A, \sO_A^*)$\footnote{Yoshioka calls these simply $Y$-sheaves, \cite[Definition 1.3]{YosModuliSheaves}}.
    In particular, non-twisted $Y$-sheaves are simply pullback of sheaves on $A$.

    Since here all $\alpha$-twisted semistable $Y$-sheaves are stable, it is enough to show that the obstruction to the deformation of a 
     stable $\alpha$-twisted $Y$-sheaf $E$ is trivial. To this end, recall that a deformation of $E$ induces a deformation of $\det E$, preserving the obstruction class (see \cite[Theorem 4.5.3]{HuyLeh}). Since the deformation of $\det E$ is unobstructed, the obstruction at $E$ is determined by the kernel of the trace map 
    \begin{equation}\label{eq:trext}
        \tr\colon \Ext^2(E,E)\to H^2(Y, \sO_Y).
    \end{equation}
    Recall that the trace map above is induced by the $H^2$ of local trace map
    $\mathbf{R}\sH om(E,E)\to \sO_Y$.
    Since $\mathbf{R}\sH om(E,E)$ is a non-twisted object in $\mathbf{D}^b(Y)$, we may write $\mathbf{R}\sH om(E, E)\simeq p^*\sE_0$ for some non-twisted complex of sheaves $\sE_0$ on $X$. 
    Then by Grothendieck duality on $A$, we have {a quasi-isomorphism:}
    \begin{equation}\label{eq:grothendieckdual}
        \mathbf{R}\Hom_A(\sO_A, \sE_0\otimes \omega_A)[2]\simeq \mathbf{R}\Hom_A(\sE_0, \sO_A)^{\vee}.
    \end{equation}
    Now since $Rp_*\sO_Y\simeq \sO_A$, we observe {that the following holds:}
    \[\mathbf{R}\Hom_A(\sO_A, \sE_0\otimes \omega_A)\simeq \mathbf{R}\Gamma_A(\sE_0)\simeq \mathbf{R}\Gamma_YR\sH om(E,E)\simeq \mathbf{R}\Hom_Y(E,E).\]
    Similarly{, we obtain the following quasi-isomorphism:}
    \[\mathbf{R}\Hom_A(\sE_0, \sO_A)\simeq \mathbf{R}\Hom_Y(\mathbf{R}\sH om(E,E),\sO_Y).\]
    Using \cite[p.\ 84, (3.15)]{HuyFM} and the equivalence  $\mathbf{R}\Gamma_Y\circ \mathbf{R}\sH om = \mathbf{R}\Hom_Y$, we deduce that
    \begin{equation}
\mathbf{R}\Hom_Y(\mathbf{R}\sH om(E,E),\sO_Y)\simeq 
\mathbf{R}\Hom_Y(E,E)^\vee
\end{equation}
    holds. Therefore, \eqref{eq:grothendieckdual} can be rewritten as
    \begin{equation}\label{eq:dualRHomEE}
        \mathbf{R}\Hom_Y(E,E)[2]\simeq \mathbf{R}\Hom_Y(E,E)^{\vee}.
    \end{equation}
    Taking degree $0$ cohomology, we obtain that
    $\Ext^2(E,E)\simeq \Hom(\Hom(E,E),\bbC)$.
    Chasing a similar series of isomorphisms, one may also argue that $H^2(Y,\sO_Y)\simeq H^0(Y,\sO_Y)^{\vee}$ holds. 
    Therefore, the trace map $\tr\colon \Ext^2(E,E)\to H^2(Y,\sO_Y)$ is 
    induced by the following commutative square
    
    \[
    \begin{tikzcd}
        \Ext^2(E,E) \rar["\tr"]\dar["\simeq"] & H^2(Y,\sO_Y)\dar["\simeq"] \\
        \Hom(E,E)^{\vee}& H^0(Y,\sO_Y)^{\vee}\lar["\tr^{\vee}", "\sim"']
    \end{tikzcd}.
    \]
    We obtain the desired isomorphism, showing that $\ker(\tr)$ is trivial.

    Finally, since $\Ext^{>2}(E,E)=0$ by (\ref{eq:dualRHomEE}), the dimension of the moduli space $M_{A,B}(v)$ is $\dim\Ext^1(E,E)$ and the formula in terms of the $B$-Mukai vector follows from Hirzebruch--Riemann--Roch \cite[Section 3.1]{YoshiokaTwistedSheaves}\footnote{Yoshioka uses a different Mukai vector than Huybrechts--Stellari, but they differ by the Hodge isometry $v\mapsto e^B\cdot v$; see also \cite[Lemma 3.1.56]{VanBreePhD}.}.\qedhere
\end{proof}

\subsection{Deforming the moduli space}\label{sec:defmoduli}

In order to prove Theorem \ref{thm:twistedYoshioka}, we deform the abelian surface and the Brauer class to another abelian surface, on which the Brauer class becomes trivial. Similar ideas appears in many papers already  (see e.g. \cite{MukaiModuliBundlesI,CaldararuPhD}). Therefore, we fix the following data:

\begin{enumerate}
    \item A polarized abelian surface $(A,H)$,
    \item A Brauer class $\overline{\alpha}\in \Br(A)=H^2_{\text{ét}}(A,\bbG_m)$ of order $k$,
    \item Lifts $\alpha\in H^2_{\text{ét}}(A,\mu_k)$ and $B\in H^2(A,\bbQ)$ of $\overline{\alpha}$. In particular, the class $\alpha$ represents a $\mu_k$-gerbe $\mathfrak{X}\to A$.
    \item A Chern character ${\rm ch}=(r,D,s)\in \widetilde{H}(A,\bbQ)$, with associated $B$-Mukai vector $v=\exp(B)\cdot {\rm ch}=(r,\theta,\chi)\in \widetilde{H}(A,B,\bbZ)$. We assume that $v$ is positive in the sense of Definition \ref{def:posvector}.
\end{enumerate}

Note that $v$ being algebraic in the $B$-twisted Mukai lattice is equivalent to having $$D\coloneqq \theta-rB\in \NS(A).$$ Fix an admissible marking $\mu\colon H^2(A, \bbZ)\to U^{\oplus3}$ of $A$, denote $h \coloneqq \mu(H)$ and $d\coloneqq \mu(D)$, and let $L \coloneqq \operatorname{Sat}_{U^{\oplus3}}(\bbZ h \oplus\bbZ d)$. Consider $\calD_{L}^+$ the period domain for $(L, h)$-polarized abelian surfaces, as defined in \cite[Lemma 2.5]{FanLaiLatticePol}. Consider an étale neighborhood of the point $\mu(H^{2,0}(A))\eqqcolon 0\in U\subset\calD_{L}^+$ such that there exists a universal family 
$$f\colon (\calA,\calH,\calD) \to U,$$
where $\calH$ and $\calD$ are divisors on $\calA$, and for every $t\in U$ we have $c_1(\calH_t)=\mu_t^{-1}(h)$ and $c_1(\calD_t)=\mu_t^{-1}(d)$.

Consider the relative Chern character $\widetilde{\rm ch}\in \oplus_{p=0}^{2}\Gamma(U,\calH^{2p}_{\rm DR})$ such that the fiber  over $0\in U$ recovers the Chern class $\ch\in \widetilde{H}(A, \bbQ)$ that we fixed, in particular, $(\widetilde{\rm ch})_0={\rm ch}$ (see e.g. \cite[Remark 4.2]{vBGJM}). 
Here for any $p\in\bbN$, $\calH^{2p}_{\rm DR}\coloneqq \mathbf{R}^{2p}f_*\Omega_{\sA/U}^{\bullet}\simeq \mathbf{R}^{2p}f_*\bbC_{\sA}$.
This relative Chern character $\widetilde{\rm ch}$ indeed exists: the bundles $\calH^0_{\rm DR}\simeq f_*\bbC_{\sA}$ and $\calH^4_{\rm DR}\simeq \mathbf{R}^4f_*\bbC_{\sA}$ are trivial, and $\calD$ gives a section of $\calH^2_{\rm DR}$ with fiber at $0\in U$ being $c_1(\calD_{0})=D$.  Since $\mathbf{R}^2f_*\mu_k$ is an étale local system, up to replacing $U$ by a finite étale cover, we can extend $\alpha$ to a relative $\mu_k$-gerbe $\widetilde{\alpha}\in H^0(U,\mathbf{R}^2f_*\mu_k)$. 
Moreover, the obstruction to lift it to a $\mu_k$-gerbe over $\calA$ vanishes étale locally, therefore we can assume that $\widetilde{\alpha}\in H^2(\calA,\mu_k)$.

\begin{lemma}\label{lem:Brauertrivialdense}
    The set of points $t\in U$ such that the Brauer class in $H^2(\calA_t,\bbG_m)$ associated to $\widetilde{\alpha}_t\in H^2(\calA_t,\mu_k)$ is trivial is dense in the analytic topology of $U$.
\end{lemma}

\begin{proof}
    This is a consequence of \cite[Proposition 5.20]{VoisinHTtwo}. See for instance \cite[Proposition 5.3]{vBGJM}.
\end{proof}

Let $\widetilde{\mathfrak{X}}\to \calA$ represent the $\mu_k$-gerbe $\widetilde{\alpha}$.
Altogether with the relative polarization $\calH$, we obtain a relative moduli space
$$\sM\coloneqq M_{\widetilde{\mathfrak{X}}/U}(\widetilde{\ch}) \to U.$$
In the next lemma we show that up to shrinking $U$, we may assume that $\calM$ is smooth over $U$. One could possibly upgrade the arguments in \cref{prop:modulitwistsmoothproj} to show this smoothness, however we give a different proof below.
\begin{lemma}\label{lem:relativemodulismooth}
     Up to shrinking $U$, the map
    $\sM \to U$
    is smooth and projective.
\end{lemma}

\begin{proof}
    Stability is an open property, see \cite[Proposition 2.1.3]{HuyLeh} (the proof is identical in the twisted setting). Therefore, up to shrinking $U$, we may assume that all fibers $\sM_t$, $t\in U$ are smooth of the same dimension and $U$ is regular. It remains to prove flatness of $\sM\to U$. For this, we use that a morphism $f\colon X\to Y$ between finite type Noetherian schemes, with $Y$ regular, is smooth around each point $x\in X$ for which $\dim \calO_{X,x}=\dim \calO_{Y,f(x)}+\dim \calO_{X_{f(x)},x}$ \cite[Lemma 2.8]{dJAlterations}.
\end{proof}

Since $\sM\to U$ is smooth, up to  replacing $U$ by a finite \'etale cover we may assume that it admits a section $s\colon U\to \sM$. We let $\sA lb\coloneqq \Pic^0_{\Pic^0_{\sM/U}/U}$ and consider the relative Albanese diagram (see e.g. \cite[Section 2.1]{Church})
\[ \begin{tikzcd}
    \sM \ar[rr,"a"] \ar[dr] & & {\sA lb} \ar[dl]\\
    & U \ar[ur, bend right, "e"'] \ar[ul, bend left, "s"] & 
\end{tikzcd}\]
with $e\colon U\to {\sA lb}$ the $0$-section. Set $\sK\coloneqq \sM\times_{\sA lb}U$ the relative Albanese fiber.

\begin{lemma}\label{lem:smoothalbanesefiber}
    The morphism $\sK\to U$ is smooth and projective.
\end{lemma}

\begin{proof}
    The dimensions of the fibers of the relative Albanese map $a\colon \sM\to \sA lb$ remain constant. Indeed, for any closed point $t\in U$, the dimension of $M_t$ is determined by the fixed Chern character and $\dim{\sA lb}_{t}=h^{1,0}(\sM_t)$ which is constant w.r.t. $t$ by Lemma \ref{lem:relativemodulismooth}.
  Since both $\sM$ and $\sA lb$ are smooth, $a$ is flat by Miracle flatness.

  Furthermore, we claim that the fibers of $a$ are smooth. To this end, we for show that for each $t\in U$, the moduli space $\sM_t$ has Kodaira dimension $0$. Indeed, by \cref{lem:relativemodulismooth},  the map $\sM\to U$ is a smooth projective morphism. Hence by \cite{TsujiPlurigena} the Kodaira dimension of the fibers remains constant over $U$. Now, by construction of $U$ as in Lemma \ref{lem:Brauertrivialdense}, there exists $t_0\in U$, such that $\sM_{t_0}$ is a moduli space of non-twisted stable sheaves on abelian surfaces. Since the Albanese fiber of $\sM_{t_0}$ is a modular Kummer-type manifold, the Kodaira dimension of $\sM_{t_0}$, and therefore of $\sM_t$ for any $t$ is zero. 
  Thus, for all $t\in U$, the map $a_t\colon \sM_t \to {\sA lb}_t$ is an \'etale fiber bundle by \cite[Theorem 8.3]{Kaw85}, and since $\calM_t$ is smooth, the morphism $a_t$ is a smooth as well.
  
  Since smoothness is an open property and $a$ is flat with smooth fibers over closed points of $U$, it is smooth \cite[\href{https://stacks.math.columbia.edu/tag/01V9}{Tag~01V9 (2)}]{stacks-project}.
  Since $\sK\to U$ is a base-change of the smooth morphism $a$, it is smooth.
\end{proof}

\begin{proof}[Proof of \cref{thm:twistedYoshioka}(1)]
Let $K\coloneqq K_{A,B}(v,H)$ be the Albanese fiber over zero of $M\to \Alb_M$, where $M=M_{A,B}(v,H)$, the moduli of stable $\alpha$-twisted sheaves with $B$-Mukai vector $v$ on the abelian surface $A$. By \cref{lem:smoothalbanesefiber}, the variety $K$ is smooth. Furthermore, by \cref{lem:Brauertrivialdense}, it deforms in a family whose special fibers are modular Kummer-type manifolds. Therefore, $K$ is a Kummer-type manifold.
\end{proof}

\subsection{Deforming the Hodge structures}
Let $\sA\to U$ be a family of abelian surfaces with $\widetilde{\mathfrak{X}}\to \sA$ a $\mu_k$-gerbe, which upon restriction as $\bbG_m$-gerbe over $\sA_t$ trivializes for an analytically dense subset of points $t\in U$ as in \cref{lem:Brauertrivialdense}.
Let $\sM\to U$ the relative moduli space $M_{\widetilde{\mathfrak{X}}/\calA}(\widetilde{{\rm ch}})$ and $\calK\to U$ the relative Albanese fiber as in Section \ref{sec:defmoduli}. 
By Theorem \ref{thm:modulitwistgeneral}, there exists a $(\beta\boxtimes \widetilde{\alpha})$-twisted universal sheaf on $\calM\times_U\calA$, where $\beta\in H^2(\calM,\bbG_m)$ is the obstruction to $\calM$ being a fine moduli space. Tensoring it by the pullback of a locally free $\beta^{-1}$-twisted sheaf on $\sM$ of rank $N$ (which always exists \cite[Theorem 1.3.5]{CaldararuPhD}), we obtain a $(0\boxtimes \widetilde{\alpha})$-twisted sheaf $\calE$ on $\sM\times_U\calA$, which satisfies the property that for each $m=[F]\in \sM_t$, $\calE|_{\{m\}\times\calA_t}\simeq F^{N}$ holds.

Now we fix an analytic open neighborhood $0\in T\subset U$, over which we restrict $\calA$, $\sM$ and $\sK$ (but we omit it from the notation). Choosing $T$ contractible, we can assume that for all $t\in T$, pullback by the closed immersion of a fiber gives an isomorphism
$$i_t^*\colon H^*(\sM,\bbQ) \xrightarrow{\sim} H^*(\sM_t,\bbQ).$$
In particular, we can consider $v$ and $B$ as cohomological classes on $\calA$, and hence on $\calA_t$ for all $t\in T$. Consider the following linear map:

\[
  \begin{aligned}
    {\Phi_t}  \colon H^{2*}(\calA_t,\bbQ) & \longrightarrow H^*(M_t,\bbQ)\\
    x & \longmapsto  \frac{1}{N}q_*(p^*x^\vee \cup v_{-B\boxtimes 0}(\calE_t^\vee))
  \end{aligned},
\]
where $\begin{tikzcd}[cramped, sep=small]
    \calA_t & \calA_t\times M_t \ar[r,"q"] \ar[l,"\ p"']& M_t
\end{tikzcd}$ are the projections, and for $x=(x_0,x_2,x_4)$ we denote $x^\vee=(x_0,-x_2,x_4)$.

\begin{proof}[Proof of Theorem \ref{thm:twistedYoshioka}(2)]
    With notations from above, a small computation shows that for all $t\in T$, the map 
\begin{equation}\label{eq:Phit}
    \Phi_{t,2}\colon H^{2*}(\calA_t,\bbQ) \xrightarrow{\Phi_t} H^*(\sM_t,\bbQ) \twoheadrightarrow H^2(\sM_t,\bbQ)
\end{equation}
sends $H^{2,0}(\sA_t)$ to $H^{2,0}(\sM_t)$, and therefore is a morphism of Hodge structure.
Restricting to $v^{\perp}\subseteq \widetilde{H}(\sA,\bbZ)$, and postcomposing with the restriction induced by the inclusion $\sK_t\into \sM_t$, we obtain
\[\Phi_{t,2}|_{v^{\perp}}\colon v^{\perp}\to H^2(\sK_t, \bbQ).\]
 
On the other hand, the map $\Phi_{t,2}$ is independent of $t$, since the rational cohomology of the fibers is constant over $T$. On any point $1\in T$ such that the Brauer class on $\sA_1$ associated to $B$ is trivial, the map $\Phi_{1,2}|_{v^\perp}=:\theta_v$ is an isometry onto $H^2(\calK_1,\bbZ)$ by \cite[Theorem 0.2]{YosModuliSheaves} (see also \cite[Section 3.2]{PerRapISVModuli}). Since $T$ contains such points by Lemma \ref{lem:Brauertrivialdense}, the result holds for all $t\in T$. 
\end{proof}

\begin{corollary}\label{coro:mukai+marking=wieneck}
The composition of 
\[\theta_{v}\colon v^\perp \to H^2(K_{A,B}(v),\bbZ)\]
with any isometry $\widetilde{H}(A,B,\bbZ)\simeq \widetilde{\Lambda}$ is a Wieneck embedding (Definition \ref{def:wieneck}).
\end{corollary}
\begin{proof}
    The result follows from the definition of Wieneck embeddings (Definition \ref{def:wieneck}). Indeed, let $1\in T$ be such that the Brauer class on $\sA_1$ associated to $B$ is trivial. 
    Then, $\Phi_{t,2}$ as in (\ref{eq:Phit}) is a deformation of $\Phi_{1,2}$.
\end{proof}

\bibliographystyle{halpha}
\bibliography{main}

\end{document}